\documentclass{amsart}
\pdfoutput=1

\usepackage{fullpage}
\usepackage{mathrsfs}
\usepackage{latexsym, amssymb}
\usepackage{stmaryrd}
\usepackage{tikz}
\usetikzlibrary{matrix,arrows,patterns,calc,decorations.markings}
\usepackage{paralist}
\usepackage{subfigure}
\usepackage[numbers]{natbib}

\usepackage{CommonMath}

\usepackage{xargs}

\newcommand{\comment}[1]{}

\newcommand {\Q}{{\bf Q}}

\renewcommand {\P}{{\bf P}}
\newcommand {\from}{{\colon}}
\newcommand{\into}{{\hookrightarrow}}
\newcommand{\onto}{\twoheadrightarrow}
\newcommand{\st}[1]{\mathscr{#1}}      
\newcommand{\sh}[1]{{\mathscr #1}}        
\newcommand{\orb}[1]{{\mathcal #1}}     
\newcommand{\cat}[1]{\mathbf{#1}}       
\renewcommand{\sp}[1]{#1}               
\newcommand{\thk}[1]{{\mathscr #1}}     
\renewcommand {\o}[1]{\overline{#1}}    
\renewcommand{\k}{{\mathbf K}}                    
\newcommand{\ideal}[1]{\langle #1 \rangle}     
\newcommand{\dual}[1]{#1^{\vee}}
\newcommand{\isom}{\stackrel\sim\longrightarrow}
\DeclareMathOperator{\Crimp}{Crimp}

\DeclareMathOperator{\Def}{Def}
\DeclareMathOperator{\Gl}{Gl}

\DeclareMathOperator{\supp}{supp}
\DeclareMathOperator{\spec}{Spec}

\DeclareMathOperator{\mult}{mult}
\DeclareMathOperator{\br}{br}
\DeclareMathOperator{\Sym}{Sym}
\DeclareMathOperator{\id}{id}
\DeclareMathOperator{\tr}{tr}
\DeclareMathOperator{\length}{length}
\DeclareMathOperator{\G}{{\mathbf G}}
\DeclareMathOperator{\A}{{\mathbf A}}

\DeclareMathOperator{\cok}{cok}
\DeclareMathOperator{\Isom}{Isom}
\DeclareMathOperator{\Aut}{Aut}

\DeclareMathOperator{\Pic}{Pic}

\DeclareMathOperator{\Hom}{Hom}
\DeclareMathOperator{\Quot}{Quot}

\newcommand{\f}[1]{\llbracket #1 \rrbracket}
\newcommand{\hens}[1]{{#1}^{\rm sh}}
\newcommand{\compl}[1]{\widehat{#1}}
\newcommand{\tw}[1]{\widetilde{#1}}
\newcommand{\orbi}{{\rm orb}}

\newcommand{\s}{\mathbf S}
\newcommand {\sm}{{\rm sm}}
\newcommand {\gen}{{\rm gen}}
\newcommand {\sing}{{\rm sing}}
\newcommand {\DM}{{\rm DM}}

\tikzset{math/.style = {execute at begin node=$, execute at end node=$}}
\tikzset{map/.style = {font=\scriptsize}}
\tikzset{equal/.style = {double distance=.15em}}
\newcommandx{\cartsquare}[3][3=.35]{
    \draw ($#1!#3!#2$) rectangle ($#2!#3!#1$);
}

\numberwithin{equation}{section}

\newcounter{thmalph}

\newtheorem{thmintro}[thmalph]{Theorem}

\title{Compactifications of Hurwitz spaces} 
\keywords{Hurwitz spaces, Hassett--Keel program}
\author{Anand Deopurkar}
\address{
  Harvard University, One Oxford Street, Cambridge MA 02139
}
\email{anandrd@math.harvard.edu}

\begin{document}
\maketitle

\begin{abstract}
We construct several modular compactifications of the Hurwitz space $H^d_{g/h}$ of genus $g$ curves expressed as $d$-sheeted, simply branched covers of genus $h$ curves. These compactifications are obtained by allowing the branch points of the covers to collide to a variable extent. They are very well-behaved if $d = 2, 3$, or if relatively few collisions are allowed. We recover as special cases the spaces of twisted admissible covers of Abramovich, Corti and Vistoli and the spaces of hyperelliptic curves of Fedorchuk.
\end{abstract}

\section{Introduction}\label{sec:intro}
A fascinating aspect of the study of moduli spaces is the exploration of their birational geometry. By varying the moduli functor, one can construct a menagerie of birational models of a moduli space. These models are not only interesting in themselves, but also provide an unprecedented opportunity to explicitly study the Mori theory of some of the most interesting higher dimensional varieties. Pioneered by Hassett and Keel, such a study for the moduli space of curves continues to be a topic of intense current research \citep{fedorchuk10:_alter_compac_modul_spaces_curves}.

We take up a similar study of a related moduli space, namely the Hurwitz space. The Hurwitz space $H^d_g$ is the moduli space of genus $g$ curves expressed as $d$-sheeted, simply branched covers of $\P^1$. These spaces have played a vital role in our understanding of the moduli of curves. They parametrize some of the most interesting loci, especially for small $d$, such as the hyperelliptic locus for $d = 2$ and the trigonal locus for $d = 3$. These loci in $M_g$ are conjectured to play a crucial role in the Hassett--Keel program. Furthermore, in many ways, the Hurwitz spaces are easier to handle than $M_g$, and it is reasonable to aspire for a fruitful Hassett--Keel program in their context.

In this paper, we lay the groundwork for constructing a number of compactifications of $H^d_g$. The standard compactification due to \citet{Harris82:_Kodair_Dimen_Of_Modul_Space_Of_Curves} (further refined by \citet*{acv:03}) parametrizes admissible covers, which are a particular kind of degenerations of simply branched covers where the branch points are forced to remain distinct. Our main idea is to explore compactifications where the branch points are allowed to coincide to a given extent. Although covers of $\P^1$ are our primary interest, we treat the case of covers of curves of arbitrary genus; this presents no additional difficulty. 

 We now describe our main results without diving into many technicalities. Fix a positive integer $d$ and non-negative integers $g$, $h$ and $b$ related by the Riemann--Hurwitz formula
\[ 2g-2 = d(2h-2) + b.\]
Let $H^d_{g/h}$ be the space of smooth genus $g$ curves expressed as $d$-sheeted, simply branched covers of smooth genus $h$ curves. In symbols, $H^d_{g/h} = \{(\phi \from C \to P)\}$, where $C$ and $P$ are smooth curves of genus $g$ and $h$ respectively, and $\phi$ is a simply branched cover of degree $d$. Let $M_{h;b}$ be the space of $b$ distinct unordered points on smooth genus $h$ curves. In symbols, $M_{h;b} = \{(P, \Sigma)\}$, where $P$ is a smooth curve of genus $h$ and $\Sigma \subset P$ a reduced divisor of degree $b$. We have a morphism $\br \from H^d_{g/h} \to M_{h;b}$ defined by 
\[\br \from (\phi \from C \to P) \mapsto (P, \br\phi).\]
Our first technical result is the construction of an unscrupulous enlargement $\st H^d_{g/h}$ of $H^d_{g/h}$ over a likewise unscrupulous enlargement of $\st M_{h;b}$ of $M_{h;b}$; we now describe both. The non-separated Artin stack $\st M_{h;b}$ is the stack of $(P, \Sigma)$, where $P$ is an at worst nodal curve of arithmetic genus $h$ and $\Sigma \subset P$ a divisor of degree $b$ supported in the smooth locus. The precise definition of $\st H^d_{g/h}$ is slightly technical, but roughly speaking, it is the stack of $(\phi \from C \to P)$, where $P$ is an orbinodal curve of arithmetic genus $h$ and $\phi$ a finite cover of degree $d$, \'etale over the nodes and the generic points of the components of $P$. There is no restriction on the singularities of $C$. The orbinodes serve to encode the admissibility criterion of \citet{Harris82:_Kodair_Dimen_Of_Modul_Space_Of_Curves}, following the idea of \citet*{acv:03}. The reader unfamiliar with this construction may imagine $P$ to be simply a nodal curve and $\phi$ an admissible cover over the nodes of $P$. As said before, the stacks $\st M_{h;b}$ and $\st H^d_{g/h}$ are non-separated enlargements of $M_{h;b}$ and $H^d_{g/h}$, respectively. They continue to be related by the branch morphism $\br \from \st H^d_{g/h} \to \st M_{h;b}$ given by
\[ \br \from (\phi \from C \to P) \mapsto (P, \br\phi).\]
\begin{thmintro}[\autoref{thm:big_hurwitz}]\label{thm:intro_big_hurwitz}
  With the above notation, $\st H^d_{g/h}$ and $\st M_{h;b}$ are algebraic stacks, locally of finite type. The morphism $\br \from \st H^d_{g/h} \to \st M_{h;b}$ is proper and of Deligne--Mumford type.
\end{thmintro}
\autoref{thm:intro_big_hurwitz} gives a recipe to construct many compactifications of $H^d_{g/h}$. Indeed, let $\orb X \subset \st M_{h;b}$ be a Deligne--Mumford substack containing $M_{h;b}$. If $\orb X$ is proper (over the base), then $\orb X \times_{\st M_{h;b}} \st H^d_{g/h}$ is a Deligne--Mumford stack containing $H^d_{g/h}$ that is also proper (over the base). In this sense, any suitable compactification of the space of branch divisors yields a corresponding compactification of the space of branched covers. Furthermore, we prove that if $\orb X$ has a projective coarse space, then so does $\orb X \times_{\st M_{h;b}} \st H^d_{g/h}$ (\autoref{thm:projectivity}).

What makes the above recipe particularly fruitful is that we know several such $\orb X$'s, leading to several compactifications of $H^d_{g/h}$. These $\orb X$'s are the spaces of weighted pointed curves of \citet{hassett03:_modul}, which we now recall. Let $\epsilon > 0$ be a rational number satisfying $b \cdot \epsilon + (2h-2) > 0$. A point of $\st M_{h;b}$ given by $(P, \Sigma)$ is called \emph{$\epsilon$-stable} if $\epsilon \cdot \mult_p\Sigma \leq 1$ for all $p \in P$ and $\omega_P(\epsilon \Sigma)$ is ample. Let $\o{\orb M}_{h;b}(\epsilon) \subset \st M_{h;b}$ be the open substack consisting of $\epsilon$-stable marked curves. Then $\o{\orb M}_{h;b}(\epsilon)$ is a proper Deligne--Mumford stack that contains $M_{h;b}$ and admits a projective coarse space. Set 
\[ \o{\orb H}^d_{g/h}(\epsilon) = \o{\orb M}_{h;b}(\epsilon) \times_{\st M_{h;b}} \st H^d_{g/h}.\]
We call points of $\o{\orb H}^d_{g/h}(\epsilon)$ \emph{weighted admissible covers} or \emph{$\epsilon$-admissible covers}. Roughly speaking, these are admissible covers where $\lfloor 1/\epsilon \rfloor$ of the branch points can coincide.\begin{thmintro}[\autoref{thm:e_admissible_covers}]\label{thm:intro_e_admissible}
  With the above notation, the stack $\o{\orb H}^d_{g/h}(\epsilon)$ of $\epsilon$-admissible covers is a proper Deligne--Mumford stack that contains $H^d_{g/h}$ as an open substack. It admits a projective coarse space $\o H^d_{g/h}(\epsilon)$ and a branch morphism to the stack $\o {\orb M}_{h;b}(\epsilon)$ of $\epsilon$-stable $b$-pointed genus $h$ curves.
\end{thmintro}
\autoref{thm:intro_e_admissible} recovers some spaces that have already appeared in literature. Plainly, the space $\o{\orb H}^d_{g/h}(1)$ is the space of twisted admissible covers of \citet*{acv:03}. In this space, the branch points are forced to remain distinct, and hence the only singularities of $C$ are its nodes over the nodes of $P$. As $\epsilon$ decreases, $\o{\orb H}^d_{g/h}(\epsilon)$ allows more and more branch points to coincide, and thus allows $C$ to have progressively nastier singularities. We highlight that these singularities need not be Gorenstein (\autoref{ex:singularities})! For $d = 2$ and $h = 0$, these spaces are the spaces of hyperelliptic curves constructed by \citet{fedorchuk10:_modul_hyp}.

In general, the local structure of $\st H^d_{g/h}$ is horrible. It may even have components other than the closure of $H^d_{g/h}$ (\autoref{ex:extra_components}). For $d = 2$ and $3$, however, $\st H^d_{g/h}$ is smooth and irreducible (\autoref{thm:smoothness_for_d23}). The geometry of the resulting compactifications of the spaces of trigonal curves is the topic of forthcoming work \citep{deopurkar12}.

The morphism $\br \from \o {\orb H}^d_{g/h}(\epsilon) \to \o {\orb M}_{h;b}(\epsilon)$ is finite for $\epsilon$ close to $1$, but not in general (\autoref{ex:extra_components}). The fibers of $\br$ parametrize ``crimps'' of a fixed $d$-sheeted cover. We analyze these fibers in detail (\autoref{sec:crimps}).

Having described the main results, let us now describe our technical motivation. Our approach is inspired by \citet*{acv:03}. We view a finite cover $\phi \from C \to P$ as a family of length $d$ schemes parametrized by $P$, or equivalently, as a map $\chi \from P \to \st A_d$, where $\st A_d$ is the `moduli stack of length $d$ schemes.' This reinterpretation allows us to use the techniques from the well-studied topic of compactifications of spaces of maps into stacks. The stack $\st H^d_{g/h}$ is thus constructed following \citet*{av02}, which explains the central role played by orbinodal curves. Note, however, that their results cannot be used directly since they deal with maps into Deligne--Mumford stacks and $\st A_d$ is not Deligne--Mumford. Nevertheless, the fact that $\st A_d$ is the quotient of an affine scheme by the general linear group allows us to extend the essential arguments without much trouble.

We carry out our constructions in a slightly more general setting than that described above. It is useful in applications to have the flexibility to fix the ramification type of some fibers of the cover. Therefore, we work in the context of covers with arbitrary branching over a divisor and prescribed branching over distinct marked points on the base. Furthermore, it is notationally easier and conceptually no harder to refrain from fixing any numerical invariants as far as we can. Therefore, instead of $\st H^d_{g/h}$ and $\st M_{h;b}$, we simply have $\st H^d$ and $\st M$.

The paper is organized as follows. In \autoref{sec:prelim}, we introduce $\st A_d$ and recall the notion of pointed orbinodal curves. In \autoref{sec:main}, we define $\st H^d$ and state the main theorem (\autoref{thm:big_hurwitz}), which we prove in \autoref{sec:proofs}. In \autoref{sec:local}, we study the local structure of $\st H^d$. In \autoref{sec:projectivity}, we prove projectivity and describe the weighted admissible cover compactifications. In \autoref{sec:crimps}, we analyze the fibers of $\br \from \st H^d \to \st M$. \autoref{sec:proofs} is by far the most technical. The crucial geometric steps, namely the valuative criteria, are contained in \autoref{sec:proper}.

\subsection*{Conventions}\label{sec:conventions}
We work over a field $\k$ of characteristic zero. All schemes are understood to be locally Noetherian schemes over $\k$. We reserve the letter $k$ for (variable) algebraically closed $\k$-fields. While working over an algebraically closed field $k$, ``point'' means  ``$k$-point,'' unless specified otherwise. An \emph{algebraic stack} or an \emph{algebraic space} is in the sense of \citet*{laumon00:_champ}. 

If $X$ is an algebraic space, and $x \to X$ a geometric point, then $O_{X,x}$ denotes the stalk of $O_X$ at $x$ in the \'etale topology and we set $X_x = \spec O_{X,x}$. The analytically inclined reader may imagine $O_{X,x}$ to be the ring of convergent power series around $x$ and $X_x$ to be a small simply-connected analytic neighborhood of $x$ in $X$. For a local ring $R$, the symbol $\hens{R}$ denotes its strict henselization and $\compl{R}$ its completion. 

The projectivization of a vector bundle $E$ is denoted by $\P E$; this is the space of one-dimensional \emph{quotients} of $E$. A morphism $X \to Y$ is \emph{projective} if it factors as a closed embedding $X \into \P E$ followed by $\P E \to Y$ for some vector bundle $E$ on $Y$.

A \emph{curve} over a scheme $S$ is a flat, proper morphism whose geometric fibers are purely one-dimensional. The source of this morphism could be a scheme, an algebraic space or a Deligne--Mumford stack; in the last case it is usually denoted by a curly letter. A curve over $S$ is connected if its geometric fibers are connected. \emph{Genus} always means arithmetic genus. By the genus of a stacky curve, we mean the genus of its coarse space. A \emph{cover} is a representable, flat, surjective morphism. The symbol $\mu_n$ denotes the group of $n$th roots of unity; its elements are usually denoted by $\zeta$.
\section{Preliminaries}\label{sec:prelim}
\subsection{The classifying stack of length $d$ schemes}\label{sec:Ad}
Consider the category $\st{A}_d$ fibered over $\cat{Schemes}$ whose objects over a scheme $S$ are $(\phi \from X \to S)$, where $\phi$ is a finite flat morphism of degree $d$. To prove that $\st{A}_d$ is indeed an algebraic stack, we consider a more rigidified version. The data of a finite flat morphism $\phi \from X \to S$ is equivalent to the data of on $O_S$ algebra $A$ which is locally free of rank $d$ as an $O_S$ module. In the rigidified version of $\st{A}_d$, we consider such algebras along with a marked $O_S$ basis. Namely, we consider the contravariant functor $\st{B}_d \from \cat{Schemes} \to \cat{Sets}$ defined by
\begin{equation*}
  \st{B}_d \from S \mapsto \left\{\parbox{.7\textwidth}{Isomorphism classes of $(A,\tau)$, where $A$ is an $O_S$ algebra and $\tau \from A \to O_S^{\oplus d}$ an isomorphism of $O_S$ modules.}\right\}.
\end{equation*}
\begin{proposition}\label{thm:based_algebras}(\citep[Proposition~1.1]{poonen08})  The functor $\st{B}_d$ is representable by an affine scheme $B_d$ of finite type.
\end{proposition}
\begin{proof}
  Let $e_1, \dots, e_d$ be the standard basis of $O_S^{\oplus d}$. Then the data $(A, \tau)$ is equivalent to an $O_S$ algebra structure on $O_S^{\oplus d}$. An $O_S$ algebra structure is specified by maps of $O_S$ modules
  \[i \from O_S \to O_S^{\oplus d}, \text{ say } 1 \mapsto \sum d_ie_i\]
  and 
  \[ m \from O_S^{\oplus d} \otimes_S O_S^{\oplus d} \to O_S^{\oplus d}, \text{ say } e_i \otimes e_j \mapsto \sum c_{ij}^k e_k.\]
  These maps make $O^{\oplus d}_S$ an $O_S$ algebra with identity $i(1)$ and multiplication $m$ if and only if the $c_{ij}^k$ and the $d_i$ satisfy certain polynomial conditions. Thus $\st{B}_d$ is representable by a closed subscheme of $\A^{d^3+d} = \A\langle c_{ij}^k, d_i\rangle$.
\end{proof}
The scheme $B_d$ admits a natural $\Gl_d$ action, which is most easily described on the functor of points. A matrix $M \in \Gl_d(S)$ acts on $B_d(S)$ by
\begin{equation}\label{eq:gl_action}
  M \from (A, \tau) \mapsto (A, M \circ \tau).
\end{equation}

\begin{proposition}\label{thm:Ad}
  $\st{A}_d$ is equivalent to the quotient $[B_d/\Gl_d]$. 
\end{proposition}
\begin{proof}
  The proof is straightforward. Consider an object $\phi \from X \to S$ in $\st{A}_d(S)$. Let $A = \phi_* O_X$. Then $A$ is an $O_S$ algebra which is locally free of rank $d$ as an $O_S$ module. Set $P = \Isom_{O_S-\text{mod}}(A, O_S^{\oplus d})$. Then $\pi \from P \to S$ is a principal $\Gl_d$ bundle, and we have a tautological isomorphism
  \[ \tau \from \pi^* A \isom O_P^{\oplus d}.\]
  The data $(\pi^*A, \tau)$ gives a map $P \to B_d$, which is visibly $\Gl_d$ equivariant. The assignment
  \[ (\phi \from X \to S) \mapsto (\pi \from P \to S, P \to B_d)\]
  defines a morphism $\st{A}_d \to [B_d/\Gl_d]$ which is easily seen to be an isomorphism.
\end{proof}

Let $\phi \from \st X_d \to \st A_d$ be the universal object. Set $\st A = \phi_* O_{\st X_d}$ and $\st L = \dual{\det \sh A}$. We have the trace map $\tr \from \st A \to O_{\st A_d}$, which pre-composed with the multiplication $\st A \otimes \st A \to \st A$ yields a map $\st A \otimes \st A \to \st O_{\st A_d}$,
or equivalently a map $\st A \to \dual{\st A}$. Taking determinants and dualizing once more, we obtain a map 
\begin{equation}\label{eq:branch_section}
  \delta \from O_{\st A_d} \to \st L^{\otimes 2}.
\end{equation}
This is the familiar discriminant construction. Let $\st E_d \subset \st A_d$ be the maximal open substack over which $\phi$ is \'etale. The following are well-known:
\begin{compactenum}
  \item $\st E_d \subset \st A_d$ is the locus where $\delta$ is invertible;
  \item $\st E_d$ is equivalent to $B\s_d$, where $\s_d$ is the symmetric group on $d$ letters.
  \end{compactenum}
  We denote the zero locus of $\delta$ in $\st A_d$ by $\Sigma_d$ and call it the \emph{universal branch locus}. We call the ideal of $\Sigma_d \subset \st A_d$ the \emph{universal discriminant}. Given a map $\chi \from S \to \st A_d$, given by a cover $\phi \from X \to S$, we denote by $\br\phi$ the pullback $\chi^*\Sigma_d$ and call it the \emph{branch locus of $\phi$}.

\comment{
  
  \begin{proposition}[The Riemann--Hurwitz Formula]\label{thm:rh}
    Let $\phi \from X \to S$ be a finite map and $B \subset S$ the branch locus.
    \begin{compactenum}
    \item $B \subset S$ is the zero locus of a section $\delta$ of the line bundle $L = \left(\dual{\det \phi_*(O_X)}\right)^{\otimes 2}$ as in \eqref{eq:branch_section}.
    \item In the case when $X$ and $S$ are complete curves and $X \to S$ is generically \'etale on every component of $S$, we have
      \[ 2\rho_a(X) - 2  = d(2\rho_a(S)-2) + \deg(B).\]
    \end{compactenum}
  \end{proposition}
  \begin{proof}
    The first statement follows since $\Sigma_d \subset \st A_d$ is defined as the zero locus of $\delta$ as in \eqref{eq:branch_section}.
    
    For the second statement, we have
    \begin{align*}
      \deg B &= 2 \deg \phi_*(O_X)\\
      &= 2 \chi \phi_*O_X - d(\chi O_S)\\
      &= 2\rho_a(X)-2 - d(2\rho_a(S)-2).
    \end{align*}
  \end{proof}

}

\subsection{Orbinodal curves}\label{sec:orbinodal}
We recall the notion of an orbinodal curve as introduced by \citet*{av02}. Our brief exposition is based on the work of \citet*{olsson07:_log}. Orbinodal curves are called ``balanced twisted curves'' in \citep{av02} and ``twisted curves'' in \citep{olsson07:_log}. In short, a `pointed orbinodal curve' is a stacky modification of a pointed nodal curve at the nodes and at the marked points. \'Etale locally near a node, it has the form 
\[[\spec k[u,v]/uv] / \mu_n,\]
where $\mu_n$ acts by $u \mapsto \zeta u, \quad v \mapsto \zeta^{-1} v$. \'Etale locally near a marked point, it has the form
\[ [\spec k[u] / \mu_n],\]
where $\mu_n$ acts by $u \mapsto \zeta u$. The formal definition follows.
\begin{definition}
  Let $S$ be a scheme. We say that the data
  \[(\orb C \to C \to S; p_1,\dots, p_n: S \to C)\]
  is a \emph{pointed orbinodal curve} if the following are satisfied.
  \begin{compactenum}
  \item $C \to S$ is a nodal curve and $p_i \from S \to C$ pairwise disjoint sections.
  \item $\orb C \to S$ is a Deligne--Mumford stack with coarse space $\orb C \to C$. The coarse space map $\orb C \to C$ is an isomorphism over the open set $C^\gen \subset C$ which is the complement of the images of $p_i$ and the singular locus of $C \to S$.
  \item Let $c \to C$ be a geometric point lying over $s \to S$. If $c$ is a node of $C_s$, then there is an \'etale neighborhood $U \to C$ of $c$, an open set $T \subset S$ containing $s$, some $t \in O_T$, and $n \geq 1$ for which we have the following Cartesian diagram
    \[
    \begin{tikzpicture}
      \matrix(m)[matrix of math nodes, row sep=2em, column sep=3em]
      {
        \orb C \times_C U & U \\
        {[\spec O_T[u,v]/(uv-t) /\mu_n]} & \spec O_T[x,y]/(xy-t^n)\\
      };
      \path[map, ->]
      (m-1-1) edge  (m-1-2)
      (m-1-1) edge node [left] {\'etale} (m-2-1)
      (m-1-2) edge node [right] {\'etale} (m-2-2)
      (m-2-1) edge (m-2-2);
      \cartsquare{(m-1-1)}{(m-2-2)};
    \end{tikzpicture},
    \]
    Here $\mu_n$ acts by $u \mapsto \zeta u$ and $v \mapsto \zeta^{-1}v$, and the map on the bottom is given by $x \mapsto u^n$ and $y \mapsto v^n$.
  \item Let $s \to S$ be a geometric point and set $c = p_i(s)$. Then there is an \'etale neighborhood $U \to C$ of $c$ and $n \geq 1$ for which we have the Cartesian diagram
    \[
    \begin{tikzpicture}
      \matrix(m)[matrix of math nodes, row sep=2em, column sep=3em]
      {
        \orb C \times_C U & U \\
        {[\spec O_S[u]/\mu_n]} & \spec O_S[x]\\
      };
      \path[map, ->]
      (m-1-1) edge (m-1-2)
      (m-1-1) edge node [left] {\'etale} (m-2-1)
      (m-2-1) edge (m-2-2)
      (m-1-2) edge node [right] {\'etale} (m-2-2);
      \cartsquare{(m-1-1)}{(m-2-2)};
    \end{tikzpicture},
    \]
    Here $\mu_n$ acts by $u \mapsto \zeta u$, and the map on the bottom is given by $x \mapsto u^n$.
  \end{compactenum}
  We abbreviate $(\orb C \to C \to S; p_1,\dots, p_n: S \to C)$ by $(\orb C \to C; p)$. A \emph{morphism} between two pointed orbinodal curves $(\orb C_1 \to C_1; p_{1j})$ and $(\orb C_2 \to C_2; p_{2j})$ is a 1-morphism $F \from \orb C_1 \to \orb C_2$ such that the induced map $F \from C_1 \to C_2$ takes $p_{1j}$ to $p_{2j}$.
\end{definition}

Although the structure of $\orb C$ is specified for \emph{some} \'etale neighborhood, it holds for any sufficiently small neighborhood. The precise statement from \citep{olsson07:_log} follows.
\begin{proposition}\label{thm:orbinodal_structure}\citep[Proposition~2.2, Definition~2.3]{olsson07:_log}
  Let $(\orb C \to C; p)$ be a pointed orbinodal curve over $S$. For a geometric point $c \to C$, set 
  \[ \hens{\orb C} = \orb C \times_{C} \spec O_{C,c}.\]
  Let $s \to S$ be the image of $c \to C$.
  \begin{compactenum}
  \item Suppose $c$ is a node of $C_s$ and $t \in O_{S,s}$ and $x, y \in O_{C,c}$ are such that $O_{C,c}$ is isomorphic to the strict henselization of $O_{S,s}[x,y]/(xy-t^n)$ at the origin. Then, for some $n \geq 1$, we have
    \[ \hens{\orb C}\cong [\spec O_{C,c}[u,v]/(uv-t, u^n-x, v^n - y) / \mu_n],\]
    where $\mu_n$ acts by $u \mapsto \zeta u,\quad v \mapsto \zeta^{-1}v$.
  \item Suppose $c = p_i(s)$ and $x \in O_{C,c}$ is such that $O_{C,c}$ is isomorphic to the strict henselization of $O_{S,s}[x]$ at the origin. Then, for some $n \geq 1$, we have
    \[ \hens{\orb C}\cong [\spec O_{C,c}[u]/(u^n-x)/\mu_n],\]
    where $\mu_n$ acts by $u \mapsto \zeta u$.
  \end{compactenum}
\end{proposition}

\section{The Big Hurwitz Stack $\st H^d$}\label{sec:main}
Fix a positive integer $d$. The goal of this section is to define the big Hurwitz stack $\st H^d$. We first define the stack $\st M$ of divisorially marked, pointed nodal curves. 

\begin{definition}\label{def:Mh}
  Define the stack $\st M$ of divisorially marked, pointed nodal curves as  the category  fibered over $\cat{Schemes}$ whose objects over $S$ are 
  \[ \st M(S) = \{ (P \to S; \Sigma; \sigma_1, \dots, \sigma_n) \},\]
  where
\begin{compactenum}
\item $P$ is an algebraic space and $P \to S$ a connected nodal curve;
\item $\Sigma \subset P$ is a Cartier divisor, flat over $S$, lying in the smooth locus of $P \to S$;
\item $\sigma_j \from S \to P$ are pairwise disjoint sections lying in the smooth locus of $P \to S$ and away from $\Sigma$. 
\end{compactenum}
\end{definition}

\begin{restatable}{proposition}{thmmhb}
  \label{thm:mhb}
  $\st M$ is a smooth algebraic stack, locally of finite type.
\end{restatable}
\begin{proof}
  Postponed to \autoref{sec:proofs}.
\end{proof}

We now define $\st H^d$. Recall our notation from \autoref{sec:Ad}:
\begin{align*}
  \st A_d &\text{ is the classifying stack of schemes of length $d$;}\\
  \st X_d \to \st A_d &\text{ is the universal scheme of length $d$;}\\
  \Sigma_d \subset \st A_d &\text{ is the universal branch locus;}\\
  \st E_d = \st A_d \setminus \Sigma_d &\text{ is the locus of \'etale covers.}
\end{align*}
\begin{definition}\label{def:big_hurwitz}
  Define the \emph{big Hurwitz stack} $\st H^d$ as the category fibered over $\cat{Schemes}$ whose objects over $S$ are
\begin{equation}
  \st H^d(S) = \{(\orb P \to P \to S;\sigma_1,\dots, \sigma_n; \chi \from \orb P \to \st A_d) \}, 
\end{equation}
where
\begin{compactenum}
\item $(\orb P \to P \to S; \sigma_1,\dots, \sigma_n)$ is a pointed orbinodal curve;
\item\label{part:representability}
  $\chi \from \orb P \to \st A_d$ is a representable morphism that maps the following to $\st E_d$: the generic points of the components of $\orb P_s$, the nodes of $\orb P_s$, and the preimages of the marked points in $\orb P_s$, for every fiber $\orb P_s$ of $\orb P \to S$.
\end{compactenum}

A morphism between $(\orb P_1 \to P_1 \to S_1; \{\sigma_{1j}\}; \chi_1 \from \orb P_1 \to \st A_d)$ and $(\orb P_2 \to P_2 \to S_2; \{\sigma_{2j}\}; \chi_2 \from \orb P_2 \to \st A_d)$ over a morphism $S_1 \to S_2$ consists of two pieces of data: $(F, \alpha)$, where
\begin{compactenum}
\item $F$ is a morphism of pointed orbinodal curves: $F \from \orb P_1 \to \orb P_2$, and
\item $\alpha$ is a 2-morphism: $\alpha \from \chi_1 \to \chi_2 \circ F$,
\end{compactenum}
such that $(F, \alpha)$ fits in a Cartesian diagram 
\begin{equation}\label{eq:Hdg_morphisms}
  \begin{tikzpicture}[math, node distance=4em, baseline=(P1.base)]
    \node (S1) {S_1}; \node (S2) [right of=S1] {S_2};
    \node (P1) [above of=S1] {\orb P_1}; 
    \node (P2) [right of=P1] {\orb P_2};
    \node (Ad) [above right of=P2] {\st A_d}; 
    \path[->, map] 
    (P1) edge node [auto] {F} (P2) 
    (P1) edge (S1)
    (S1) edge (S2) 
    (P2) edge (S2)
    (P1) edge [bend left=40] node (Chi1) [left] {\chi_1} (Ad)
    (P2) edge [bend right=20] node (Chi2) [right] {\chi_2} (Ad)
    (Chi1) edge [dashed,->, shorten <=.6em, shorten >=.6em] 
    node[auto]{\alpha} (Chi2);
    \cartsquare{(P1)}{(S2)};
  \end{tikzpicture}
\end{equation}
We abbreviate $(\orb P \to P \to S; \sigma_1,\dots, \sigma_n; \chi \from \orb P \to \st A_d)$ by $(\orb P \to P; \sigma; \chi)$.
\end{definition}

\begin{remark}
  The careful reader may wonder what happened to the $2$-morphisms between the $1$-morphisms from $\orb P_1$ to $\orb P_2$. After all, the objects of $\st H^d$ involve stacks, which makes it, \emph{a priori}, a $2$-category. However, by \citep[Lemma~4.2.3]{av02}, the $2$-automorphism group of any $1$-morphism $\orb P_1 \to \orb P_2$ is trivial. Thus, $\st H^d$ is equivalent to a $1$-category \citep[Proposition~4.2.2]{av02}. What this means explicitly is that we treat two morphisms given by $(F, \alpha)$ and $(F',\alpha')$ as \emph{the same} if they are related by a $2$-morphism between $F$ and $F'$.
\end{remark}
\begin{remark}
  Let us explain the condition of representability of $\chi$ (\autoref{def:big_hurwitz}~\eqref{part:representability}). A morphism between two Deligne--Mumford stacks $F \from \st X \to \st Y$ is representable if and only if for every geometric point $x \to \st X$, the induced map of automorphism groups $\Aut_x(\sh X) \to \Aut_{F(x)}(\sh Y)$ is injective \citep[Lemma~4.4.3]{av02}. Thus the representability of $\chi$ means that the stack structure on $\orb P$ is the minimal one that affords a morphism to $\st A_d$.
\end{remark}
\begin{remark}
  Let us explain the role played by the orbinodes. Consider a local piece of an orbinodal curve near a node; say $\orb U = [\spec\left(k[u,v]/uv\right) / \mu_n]$ and an \'etale cover $\orb C \to \orb U$. Observe that the induced map on the coarse spaces $C \to U$ is precisely an admissible cover in the sense of \citet{Harris82:_Kodair_Dimen_Of_Modul_Space_Of_Curves}. In this way, the orbinodes provide a way to deal with the admissibility condition.
\end{remark}
\begin{remark}\label{rem:marked_points}
  Let us explain the role played by the marked points. Consider a local piece of an orbinodal curve near a marked point; say $\orb U = [\spec k[u]/\mu_n]$. The morphism $\chi$ maps such a piece into $\st E_d \cong B\s_d$, corresponding to an \'etale cover $\orb C \to \orb U$. Note that in contrast to the fundamental group of a small piece of a schematic curve, the fundamental group of the stacky curve $\orb U$ is not trivial; it is precisely $\mu_n$. Thus, $\orb C \to \orb U$ may be a non-trivial \'etale cover, specified by the monodromy
  \[ \Aut_0(\orb U) = \mu_n \to \Aut_0(B\s_d) = \s_d.\]
  The condition of representability implies that this monodromy map is injective. On the level of coarse spaces, we thus get a cover $C \to U$ with monodromy around $0$ given by an element of order $n$ in $\s_d$. By taking the open and closed substack of $\st H^d$ where $\Aut_{\sigma_i}(\orb P)$ has order $n$, we in effect impose the condition that the monodromy of $\orb C \to \orb P$ around $\sigma_i$ is a permutation $\pi \in \s_d$ of order exactly $n$. By further restricting to the open and closed substack where $\pi$ has a specific cycle structure, we can fully prescribe the monodromy. In this way, we can get moduli spaces of covers with marked fibers of prescribed ramification type.
\end{remark}

It is useful to have a formulation of $\st H^d$ purely in terms of finite covers. Since a map to $\st A_d$ is nothing but a finite cover of degree $d$, we see that $\st H^d$ may be equivalently described as the category whose objects over a scheme $S$ are

\begin{equation}
  \{(\orb P \to P \to S;  \sigma_1, \dots, \sigma_n; \phi \from \orb C \to \orb P)\},
\end{equation}
where
\begin{compactenum}
\item $(\orb P \to P \to S; \sigma_1, \dots, \sigma_n)$ is a pointed orbinodal curve;
\item $\phi$ is a finite cover of degree $d$, \'etale over the following: the generic points of the components of $\orb P_s$, the nodes of $\orb P_s$, and the preimages of the marked points in $\orb P_s$, for every fiber $\orb P_s$ of $\orb P \to S$.
\item Furthermore, the following condition is satisfied: for every open subset $\orb U \subset \orb P \setminus \br\phi$, the morphism $\orb U \to B\s_d$ corresponding to the \'etale cover $\orb C|_{\orb U} \to \orb U$ is representable.
\end{compactenum}

In this form, a morphism from  $(\orb P_1 \to P_1 \to S_1; \sigma_{1j}; \phi_1 \from \orb C_1 \to \orb P_1)$ to $(\orb P_2 \to P_2 \to S_2; \sigma_{2j}; \phi_2 \from \orb C_2 \to \orb P_1)$ is given by $(F, G)$ where $F \from \orb P_1 \to \orb P_2$ is a morphism of pointed orbinodal curves and $G \from \orb C_1 \to \orb C_2$ a morphism over $F$ such that there is a Cartesian diagram
\[
\begin{tikzpicture}[math, node distance=4em, baseline=(P1.base)]
  \node (S1) {S_1}; \node (S2) [right of=S1] {S_2};
  \node (P1) [above of=S1] {\orb P_1}; 
  \node (P2) [right of=P1] {\orb P_2};
  \node (C1) [above of=P1] {\orb C_1}; 
  \node (C2) [above of=P2] {\orb C_2}; 
  \path[->, map] 
  (P1) edge node [auto] {F} (P2) 
  (C1) edge node [auto] {G} (C2)
  (P1) edge (S1)
  (P2) edge (S2)
  (C1) edge (P1)
  (C2) edge (P2)
  (S1) edge (S2) 
  ;
  \cartsquare{(P1)}{(S2)};
  \cartsquare{(C1)}{(P2)};
\end{tikzpicture}.
\]
We abbreviate $(\orb P \to P \to S; \sigma_1, \dots, \sigma_n; \phi \from \orb C \to \orb P)$ by $(\orb P \to P; \sigma; \phi)$. We use the formulation of $\st H^d$ in terms of maps to $\st A_d$ or in terms of finite covers depending on  whichever is convenient.

The two stacks $\st H^d$ and $\st M$ are related by the branch morphism, which we now define.  Consider an object $(\orb P \to P \to S; \sigma_1, \dots, \sigma_n; \phi \from \orb C \to \orb P)$ in $\st H^d(S)$. Identify $\br\phi$ with its image in $P$ (it is anyway disjoint from the stacky points of $\orb P$). Then $\br\phi \subset P$ is an $S$-flat Cartier divisor.  The branch morphism $\br \from \st H^d \to \st M$ is defined by 
  \[ \br \from (\orb P \to P \to S; \sigma_1, \dots, \sigma_n; \phi \from \orb  C \to \orb P) \mapsto (P \to S; \br\phi; \sigma_1, \dots, \sigma_n).\]

\begin{restatable}[Main]{theorem}{thmbighurwitz}
  \label{thm:big_hurwitz}
  $\st H^d$ is an algebraic stack, locally of finite type. The morphism 
  \[\br \from \st H^d \to \st M\]
  is proper and representble by Deligne--Mumford stacks.
\end{restatable}

\autoref{thm:big_hurwitz} is motivated by the treatment of Hurwitz spaces as spaces of maps into a suitable stack by \citet*{acv:03}, relying on the work of \citet*{av02}. The proof of the main theorem in \citep{av02} is quite involved. However, thanks to the advancement of technology related to stacks, we can give a relatively short and conceptual proof. There is a very general result for the existence of $\Hom$ stacks due to \citet*{aoki05:_hom}, but it is not suitable for our purpose because it does not yield the required finiteness properties. 

  A generalization of \autoref{thm:big_hurwitz} where $\st A_d$ is replaced by a suitable global quotient $[U/G]$ seems plausible. This would also generalize the construction by \citet*{ciocan-fontanine11:_stabl_git}. However, this is beyond the scope of the present work.

\section{Proof of the Main Theorem}\label{sec:proofs}
This section is devoted to proving \autoref{thm:big_hurwitz}. The proof is broken down into parts.

\subsection{That $\st M$ is a smooth algebraic stack, locally of finite type}\label{sec:Mh}
This result is essentially \citep[Lemma~5.1]{olsson07:_log}. We sketch a proof for completeness.

\begin{lemma}\label{thm:etale_local_proj}
\citep[Proposition~2.1]{hall10:_modul_singul_curves}
Let $\pi \from P \to S$ be a nodal curve, where $P$ is an algebraic space and $S$ a scheme. Let $s \to S$ be a geometric point. Then there is an \'etale neighborhood $T \to S$ of $s$ such that $\pi_T \from P_T \to T$ is projective.
\end{lemma}
\begin{proof}
  Pick points $x_1, \dots, x_n \in P_s$ in the smooth locus such that the Cartier divisor $x_1 + \dots + x_n$ is ample on $P_s$. Since $P \to S$ is smooth along the chosen points $x_i$, there is an \'etale neighborhood $T \to S$ of $s$ such that each $x_i$ extends to a section $\sigma_i$ of $\pi_T \from P_T \to T$. By passing to a Zariski open, if necessary, assume that the sections map to the smooth locus of $\pi_T$. Then the divisor $\sigma_1(T) + \dots + \sigma_n(T)$ is a Cartier divisor on $P_T$ which is ample on the fiber over $s$. Again, by passing to a Zariski open, if necessary, we get a $\pi_T$-ample divisor. Hence $\pi_T$ is projective.
\end{proof}
We are ready to prove \autoref{thm:mhb}, which we recall for convenience.
\thmmhb*
\begin{proof}
  Let $\st M^{b, n} \subset \st M$ be the subcategory where the degree of the marked divisor is $b$ and the number of marked points is $n$. It suffices to prove that $\st M^{b, n}$ is an algebraic stack, locally of finite type. For brevity, set $\st U = \st M^{0,0}$.

Clearly, the obvious forgetful morphism $\st M^{b,n} \to \st U$ is representable by smooth algebraic spaces of finite type. Hence, it suffices to prove that $\st U$ is an algebraic stack, locally of finite type.
  
That $\st U$ is a stack over $\cat{Schemes}$ follows from standard descent theory; it will not be reproduced here. Note, however, that it is important to allow algebraic spaces (and not merely schemes) $P \to S$ in the definition of $\st M$.

We now prove that $\st U$ is algebraic. Recall that this means the following two conditions:
  \begin{compactenum}
  \item\label{diagonal} the diagonal $\st U \to \st U \times \st U$ is representable by separated algebraic spaces of finite type;
  \item\label{atlas} $\st U$ admits a smooth, surjective morphism from a scheme, locally of finite type.
  \end{compactenum}
  For \eqref{diagonal}, we must check that given two objects $P_i \to S$ of $\st U$, for $i = 1, 2$, the sheaf $\Isom_S(P_1, P_2)$ on $S$ is representable by a separated algebraic space of finite type. It suffices to check this \'etale locally on $S$. Also, this is well-known if $P_i \to S$ are projective. Thanks to \autoref{thm:etale_local_proj}, the general case follows.
  
  For \eqref{atlas}, it suffices to exhibit a smooth, surjective map to $\st U$ from an algebraic stack, which is itself locally of finite type. Denote by $\st M_{g,k}^{\DM}$ the stack of Deligne--Mumford stable curves of genus $g$ with $k$ marked points. This is a smooth algebraic stack of finite type. The forgetful morphism $\st M_{g,k}^{\DM} \to \st U$ is easily seen to be smooth, and the morphism from the disjoint union
  \[ \bigsqcup_{k \geq 0, g \geq 0} \st M_{g,k}^{\DM} \to \st U \]
  is surjective. We conclude \eqref{atlas}. Finally, the smoothness of $\st U$ follows from the smoothness of $\st M_{g,k}^{\DM}$.
\end{proof}

\subsection{That $\br \from \st H^d \to \st M$ is an algebraic stack, locally of finite type}\label{sec:algebraic}

The overall strategy is to work our way up from $\st M$ to $\st H^d$ via a series of intermediate algebraic stacks. We first introduce some covenient notation. Denote by $\st M^{b,*}$ (resp.\ $\st M^{*,n}$, $\st M^{b,n}$) the open substack of $\st M$ where the marked divisor has degree $b$ (resp.\ there are $n$ marked points, degree $b$ and $n$ marked points).

The first intermediate step is the stack of pointed orbinodal curves. Let $\st M^{\orbi}$ be the category over $\cat{Schemes}$ whose objects over $S$ are pointed orbinodal curves $(\orb P \to P \to S; \sigma)$. Denote by $\st M^{\orbi\leq N}$ the subcategory of $\st M^\orbi$ where the order of the automorphism groups at the points of the orbinodal curve is bounded above by $N$. There is a morphism $\st M^{\orbi} \to \st M^{0,*}$ given by
\[ (\orb P \to P \to S; \sigma)  \to (P \to S; \sigma).\]
We quote, without proof, a theorem of \citet{olsson07:_log}.
\begin{theorem}\label{thm:orbinodal_stack}\citep[Theorem~1.9, Corollary~1.11]{olsson07:_log}
  $\st M^\orbi$ and $\st M^{\orbi\leq N}$ are smooth algebraic stacks, locally of finite type. $\st M^{\orbi\leq N}$ is an open substack of $\st M^\orbi$. The morphism $\st M^{\orbi\leq N} \to \st M^{0,*}$ is representable by Deligne--Mumford stacks of finite type.
\end{theorem}

We have a morphism $\st H^d \to \st M^\orbi$ given by
\[ (\orb P \to P; \sigma; \orb C \to \orb P) \mapsto (\orb P \to P; \sigma).\]
Define categories $\st FinCov^d$ and $\st Vect^d$ fibered over $\cat{Schemes}$ as follows
\begin{align*}
  \st FinCov^d(S) &= \{ (\orb P \to P \to S;\sigma; \phi \from \orb C \to \orb P), \text{ where $\phi$ is finite, flat of degree $d$} \}, \\
  \st Vect^d(S)  &= \{ (\orb P \to P\to S; \sigma; \orb F), \text{ where $\orb F$ is locally free of rank $d$ on $\orb P$}\}.
\end{align*}
In both definitions, $(\orb P \to P \to S; \sigma)$ is a pointed orbinodal curve. We have morphisms
\begin{equation}\label{eqn:reductions}
 \st H^d \to \st FinCov^d \to \st Vect^d \to \st M^\orbi.
\end{equation}
Indeed, the first is obvious; the second is given by 
\[ (\orb P \to P;\sigma; \phi \from \orb C \to \orb P) \mapsto (\orb P \to P; \sigma; \phi_* O_{\orb C});\]
and the last by 
\[ (\orb P \to P; \sigma; \orb F) \mapsto (\orb P \to P; \sigma).\]
We analyze each morphism in \eqref{eqn:reductions} one by one.

Before we proceed, we need some results on the structure of orbinodal curves. We first recall the notion of a \emph{generating sheaf} on a Deligne--Mumford stack from \citep[\S~5.2]{kresch09:_delig_mumfor}. Let $\orb X$ be a Deligne--Mumford stack with coarse space $\rho \from \orb X \to X$. A locally free sheaf $\orb E$ on $\orb X$ is a \emph{generating sheaf} if for every quasi coherent sheaf $\orb F$, the morphism
\[ \rho^*\rho_*(\sh Hom_{\orb X}(\orb E, \orb F) \otimes_{O_{\orb X}} \orb E) \to \orb F\]
  is surjective. Equivalently, $\orb E$ is a generating sheaf if and only if for every point $x$ of $\orb X$, the representation of $\Aut_x(\orb X)$ on the fiber of $\orb E$ at $x$ contains every irreducible representation of $\Aut_x(\orb X)$.
\begin{proposition}\label{thm:orbinodal_global_quot}
  Let $S$ be a scheme and $(\orb P \to P \to S; \sigma)$ a pointed orbinodal curve. There is a scheme $T$ and a surjective \'etale morphism $T \to S$ such that
  \begin{compactenum}
  \item $\orb P_T$ admits a finite, flat morphism from a projective scheme $Z$;
  \item $\orb P_T$ is the quotient of a quasi projective scheme by a linear algebraic group;
  \item $\orb P_T$ admits a generating sheaf.
  \end{compactenum}
\end{proposition}
\begin{proof}
  The first statement is due to \citet[Theorem~1.13]{olsson07:_log}. The existence of a finite flat cover $Z \to \orb P_T$ implies that $\orb P_T$ is the quotient of an algebraic space $Y$ by the action of a linear algebraic group by \citep[Theorem~2.14]{edidin01:_brauer}. We may assume that $T$ is affine and, by \autoref{thm:etale_local_proj}, that $P_T$ is projective over $T$. Then $P_T$ is quasi-projective. In this case, $Y$ can be proved to be quasi-projective \citep[Remark~4.3]{kresch09:_delig_mumfor}. Finally, since $\orb P$ is a quotient stack with a quasi projective coarse space, the third statement follows directly from \citep[Theorem~5.3]{kresch09:_delig_mumfor}.
\end{proof}

\begin{proposition}\label{thm:vect_to_orb}
  $\st Vect^d \to \st M^\orbi$ is an algebraic stack, locally of finite type.
\end{proposition}
\begin{proof}
  Let $S$ be a scheme and $(\orb P \to P \to S; \sigma)$ an object of $\st M^\orbi$. We must prove that the category of vector bundles of rank $d$ on $\orb P$ is an algebraic stack, locally of finite type. It suffices to prove this after passing to an \'etale cover of $S$. By \autoref{thm:orbinodal_global_quot}, we can assume that $\orb P \to S$ admits a generating sheaf and by \autoref{thm:etale_local_proj}, that $P \to S$ is projective. Now it can be shown that the stack $\st Coh_{\orb P/S}$ of coherent sheaves on $\orb P$, flat over $S$, is an algebraic stack, locally of finite type. A smooth atlas is given by the $\Quot$ schemes of \citet{olsson03:_quot_delig_mumfor}. We omit the details; see the pre-print by \citet[\S~2.1]{nironi08:_modul_spaces_semis_sheav_projec} for a complete proof. Clearly, the stack of vector bundles of rank $d$ on $\orb P$ is an open substack of $\st Coh_{\orb P/S}$.
\end{proof}

\begin{proposition}\label{thm:cover_to_vb}
  $\st FinCov^d \to \st Vect^d$ is representable by algebraic spaces of finite type.
\end{proposition}
For the proof, we need two easy lemmas.
\begin{lemma}\label{thm:cohom_base_change}
  Let $S$ be an affine scheme and $\orb X \to S$ be a proper Deligne--Mumford stack with coarse space $\rho \from \orb X \to X$, where $X$ is a scheme. Let $\orb F$ be a coherent sheaf on $\orb X$, flat over $S$. Then, there is a finite complex $M_\bullet$ of locally free sheaves on $S$:
  \[ M_0 \to M_1 \to \dots \to M_n\]
  such that for every $f \from T \to S$, we have natural isomorphisms
  \[ H^i(f^*M_\bullet) \isom H^i(\orb X_T, \orb F_T);\]
\end{lemma}
\begin{proof}
  Let $F = \rho_* \orb F$. Then $F$ is a coherent sheaf on $X$, flat over $S$. Since $X$ is a proper scheme over $S$, the standard theorem on cohomology and base change for schemes \citep[\S II.5]{mumford08:_abelian}, gives a finite complex of locally free sheaves $M_\bullet$ with natural isomorphisms
  \begin{equation}\label{eqn:M_computes_on_coarse_space}
    H^i(f^*M_\bullet) \isom H^i(X_T, F_T).
  \end{equation}
  Now, the map $\rho_T \from \orb X_T \to X_T$ is the map to the coarse space. Since maps to the coarse spaces are cohomologically trivial for quasi-coherent sheaves, we have ${\rho_T}_* (\orb F_T) = F_T$ and a natural identification
  \begin{equation}\label{eqn:cohom_on_coarse_space_same_as_original}
     H^i(X_T, F_T) = H^i(\orb X_T, \orb F_T).
  \end{equation}
  Combining \eqref{eqn:M_computes_on_coarse_space} and \eqref{eqn:cohom_on_coarse_space_same_as_original}, we obtain the result.
\end{proof}
\begin{lemma}\label{thm:represent_global_sections}
  Let $\orb X \to S$ and $\orb F$ be as in \autoref{thm:cohom_base_change}. Then the contravariant functor from $\cat{Schemes}_S$ to $\cat{Sets}$ defined by
  \[ (f \from T \to S) \mapsto H^0(\orb X_T, \orb F_T)\]
  is representable by an affine scheme $Sect_{\orb F/S}$ over $S$.
\end{lemma}
When no confusion is likely, we denote $Sect_{\orb F/S}$ by $Sect_{\orb F}$.
\begin{proof}
  Let $M_\bullet$ be as in \autoref{thm:cohom_base_change}. Let $T_i = \spec_S (\Sym^*(\dual{M_i}))$ be the total spaces of the vector bundles $M_i$ (we only care about $i = 0, 1$). Then $T_i$ are vector bundles over $S$ and we have a morphism $T_0 \to T_1$. Let $Sect_{\orb F} \subset T_0$ be the scheme theoretic preimage of the zero section of $T_1$. From the natural isomorphism
  \[ H^0(f^*M_\bullet) \isom H^0(\orb X_T, \orb F_T),\]
  it is easy to see that $Sect_{\orb F}$ represents the desired functor.
\end{proof}

We now have the tools to prove \autoref{thm:cover_to_vb}.
\begin{proof}[Proof of \autoref{thm:cover_to_vb}]
  Let $S$ be a scheme and $S \to \st Vect^d$ a morphism given by the object $(\orb P \to P \to S; \sigma; \orb F)$ of $\st Vect^d(S)$. We must prove that $\st FinCov^d \times_{\st Vect^d} S$ is an algebraic space of finite type. It suffices to prove this after passing to an \'etale cover of $S$. So, assume that $S$ is affine and $P$ is projective over $S$. By an \emph{$O_{\orb P}$-algebra structure} on $\orb F$, we mean a pair $(i,m)$, where $i \from O_{\orb P} \to \orb F$ and $m \from \orb F \otimes \orb F \to \orb F$ are morphisms of $O_{\orb P}$ modules that make $\orb F$ a sheaf of $O_{\orb P}$-algebras. Let $\st Alg_{\orb F}$ be the stack of $O_{\orb P}$-algebra structures on $\orb F$. The operation of taking the spectrum gives an equivalence
  \[\st Alg_{\orb F} \isom \st FinCov^d \times_{\st Vect^d} S.\]
  Now, an algebra structure on $\orb F$ is determined by a global section (corresponding to $i$) of $\orb F$ and one (corresponding to $m$) of $\sh Hom(\orb F \otimes \orb F, \orb F)$ subject to the conditions
  \begin{align*}
    m \circ (i \otimes \id) &=     m \circ (\id \otimes 1) = \id \quad &\text{(multiplicative identity)}\\
    m \circ {\rm sw} &= m \quad &\text{(symmetry)}\\
    m \circ (\id \otimes m) &= m \circ (m \otimes id) \quad &\text{(associativity)}, 
  \end{align*}
  where $\rm sw \from \orb F \otimes \orb F \to \orb F \otimes \orb F$ is the switch $x \otimes y \mapsto y \otimes x$. Each of these equations can be interpreted as the vanishing (agreeing with the zero section) of a morphism from $Sect_{\orb F} \times_S Sect_{\sh Hom(\orb F, \orb F \otimes \orb F)}$ to a suitable $Sect$ space. For example, the equality 
  \[ m \circ (\id \otimes 1) = \id\]
  can be phrased as the vanishing of the morphism 
  \[ Sect_{\orb F} \times_S Sect_{\sh Hom(\orb F, \orb F \otimes \orb F)} \to Sect_{\sh Hom(\orb F, \orb F)}\] defined by
  \[ (i, m) \mapsto m \circ (i \otimes \id) - \id.\]
  Thus, $\st Alg_{\orb F}$ is represented by the closed subscheme of $Sect_{\orb F} \times_S Sect_{\sh Hom(\orb F, \orb F \otimes \orb F)}$ defined by vanishing of the equations given by the conditions above.
\end{proof}

We finish the final piece of \eqref{eqn:reductions}.
\begin{proposition}\label{thm:hdg_to_cov}
  $\st H^d \to \st FinCov^d$ is an open immersion.
\end{proposition}
\begin{proof}
  Let $S$ be a scheme and $S \to \st FinCov^d$ a morphism corresponding to $(\orb P \to P \to S; \sigma; \phi \from \orb C \to \orb P)$. Let $\pi \from \orb P \to S$ be the projection. Denote by $\Sigma \subset P$ the image in $P$ of the branch divisor of $\phi$. Clearly, the locus $S_1 \subset S$ over which $\Sigma$ is disjoint from the singular locus of $P \to S$ and the sections $\sigma_i$ is an open subscheme. Over $S_1$, the Cartier divisor $\Sigma \subset P$ does not contain any components of the fibers and hence it is $S_1$-flat.

  Let $\chi \from \orb P \to \st A_d$ be the morphism corresponding to the degree $d$ cover $\orb C \to \orb P$. Let $\orb I_{\chi} \to \orb P$ be the inertia stack of $\chi$. Then $\orb I_{\chi} \to \orb P$ is a representable finite morphism. The set $Z \subset \orb P$ over which $\orb I_{\chi}$ has a fiber of cardinality higher than one is a closed subset and its complement is exactly the locus where $\chi$ is representable. Let $S_2 = S_1 \cap (S \setminus \pi(Z))$.

  Then, by definition, $\st H^d \times_{\st FinCov^d} S = S_2$, which is an open subscheme of $S$.
\end{proof}

We have finished the first part of the proof of \autoref{thm:big_hurwitz}. 
\begin{proposition}\label{thm:algebraic}
  The morphism $\br \from \st H^d \to \st M$ is an algebraic stack, locally of finite type.
\end{proposition}
\begin{proof}
  The forgetful morphism $\st M \to \st M^{0,*}$ is representable by algebraic spaces of finite type. Hence, it suffices to show that $\st H^d \to \st M^{0,*}$ is an algebraic stack, locally of finite type. We have the sequence of morphisms
  \begin{equation}\label{eqn:ladder}
    \st H^d \to \st FinCov^d \to \st Vect^d \to \st M^\orbi \to \st M^{0,*}.
  \end{equation}
  Starting from the right, \autoref{thm:orbinodal_stack}, \autoref{thm:vect_to_orb}, \autoref{thm:cover_to_vb} and \autoref{thm:hdg_to_cov} imply that each of the morphisms above is an algebraic stack, locally of finite type. Hence, so is their composite.
\end{proof}

\subsection{That $\br \from \st H^d \to \st M$ is of finite type}\label{sec:fin_type}

The strategy in this section is to study \eqref{eqn:ladder} more carefully and trim down the intermediate stacks so that they are of finite type. 
\begin{proposition}
  The morphism $\st H^d \to \st M^\orbi$ factors through the open substack $\st M^{\orbi\leq N}$ for any $N \geq d!$.
\end{proposition}
\begin{proof}
  Take an object $(\orb P \to P \to S;\sigma;\chi)$ of $\st H^d$. Let $p$ be a point of $\orb P$ which is either a node or a marked point in its fiber. Then $\chi$ maps a neighborhood of $p$ into $\st E_d \cong B\s_d$. Since $\chi$ is required to be representable, we have
  \[ \Aut_p(\orb P) \into \Aut_{\chi(p)}(B\s_d) = \s_d.\]
  In particular, the size of $\Aut_p(\orb P)$ is at most $d!$.
\end{proof}

Recall that $\st M^{b,*} \subset \st M$ is the open and closed substack where the marked divisor has degree $b$. Set 
\[ \st H^d_b = \st M^{b,*} \times_{\st M} \st H^d,\]
and denote by $\st Vect^d_{l, N}$ the open substack of $\st Vect^d$ parametrizing vector bundles of fiberwise degree $l$ and $h^0 \leq N$.
\begin{proposition}
  The morphism $\st H^d_b \to \st Vect^d$ factors through the open substack $\st Vect^d_{l,N}$ for $l = -b/2$ and any $N \geq d$.
\end{proposition}
\begin{proof}
  Consider a geometric point $(\orb P \to P; \sigma; \phi \from \orb C \to \orb P)$ of $\st H^d_b$. Then, the branch divisor of $\phi$, which is a section of $\det{\phi_*O_{\orb C}}^{\otimes(-2)}$, has degree $b$. Hence $\phi_* O_{\orb C}$ has degree $l$. Furthermore, since $\orb C$ is a reduced curve which is a degree $d$ cover of the connected curve $\orb P$, we must have $h^0(\phi_*O_{\orb C}) \leq d$.
\end{proof}

\begin{proposition}\label{thm:vec_fin_type}
  The morphism $\st Vect^d_{l,N} \to \st M^\orbi$ is of finite type.
\end{proposition}
For the proof, we need some results about the boundedness of families of sheaves on Deligne--Mumford stacks. Let $S$ be an affine scheme and $\orb X \to S$ a Deligne--Mumford stack with coarse space $\rho \from \orb X \to X$, and a generating sheaf $\orb E$. Let $O_X(1)$ be an $S$-relatively ample line bundle on $X$. Let $U$ be an $S$-scheme, not necessarily of finite type, and $\orb F$ a sheaf on $\orb X_U$.  We say that the family of sheaves $(\orb X_U, \orb F)$ is \emph{bounded} if there is an $S$-scheme $T$ of finite type and a sheaf $\orb G$ on $\orb X_T$ such that every geometric fiber $(\orb X_u, \orb F_u)$ appearing in $(\orb X_U, \orb F)$ over $U$ appears in $(\orb X_T, \orb G)$ over $T$. In this case, we say that $(\orb X_T, \orb G)$ \emph{bounds} $(\orb X_U, \orb F)$.

Set \[ F_{\orb E}(-) = \rho_* \sh Hom_{\orb X}(\orb E, -).\] Then $F_{\orb E}$ takes exact sequences of quasi-coherent sheaves on $\orb X$ to exact sequences on of quasi-coherent sheaves on $X$, because $\rho_*$ is cohomologically trivial.
\begin{lemma}
  In the above setup, if the family $(X_U, F_{\orb E}(\orb F))$ is bounded, then the family $(\orb X_U, \orb F)$ is also bounded.
 \end{lemma}
 \begin{proof}
   Since $F_{\orb E}(\orb F)$ is bounded, we have a surjection
   \[O_X(-M)^{\oplus N} \otimes_S O_U \onto F_{\orb E}(\orb F)\]
   for large enough $M$ and $N$. Since $\orb E$ is a generating sheaf, this gives a surjection
   \begin{equation}\label{eqn:sur1}
     \orb E \otimes_{\orb X} O_{X}(-M)^{\oplus N} \otimes_S O_U \onto \orb F.
   \end{equation}
   Let $\orb K$ be the kernel. Then, $(X_U, F_{\orb E}(\orb K))$ is also  bounded, and by the same reasoning as above, we have a surjection
   \begin{equation}\label{eqn:sur2}
     \orb E \otimes_{\orb X} O_X(-M')^{\oplus N'} \otimes_S O_U \onto \orb K
   \end{equation}
   for large enough $M'$ and $N'$. Combining \eqref{eqn:sur1} and \eqref{eqn:sur2}, $\orb F$ can be expressed as the cokernel
   \[ \orb E \otimes_{\orb X}O_X(-M')^{\oplus M'} \otimes_S O_U \to \orb E \otimes_{\orb X}O_X(-M)^{\oplus M} \otimes_S O_U \onto \orb F.\]
   Set
   \[ \orb H = \sh Hom_{\orb X}\left(\orb E \otimes_{\orb X}O_X(-M')^{\oplus N'}, \orb E \otimes_{\orb X} O_X(-M)^{\oplus N}\right),\]
   and $T = Sect_{\orb H/S}$. By \autoref{thm:represent_global_sections}, $T \to S$ is of finite type. Letting $\orb G$ be the cokernel of the universal homomorphism on $\orb X_T$, we see that $(\orb X_T, \orb G)$ bounds $(\orb X_U, \orb F)$.
\end{proof}
\begin{remark}\label{rem:b_for_curves}
  In the case of a curve $\orb X \to S$, the family $(X_U, F_{\orb E}(\orb F))$ is bounded if the degree, rank and $h^0$ of $F_{\orb E}(\orb F)_u$ are bounded for $u \in U$.
\end{remark}

We now have the tools to prove \autoref{thm:vec_fin_type}.
\begin{proof}[Proof of \autoref{thm:vec_fin_type}]
  Let $S$ be a connected affine scheme and $S \to \st M^\orbi$ a morphism given by the pointed orbinodal curve $(\orb P \stackrel\rho\to P \to S; \sigma)$. We must prove that $\st Vect^d_{l,N} \times_{\st M^\orbi} S \to S$ is of finite type. After passing to an \'etale cover of $S$ if necessary, assume that
  \begin{compactenum}
  \item $P \to S$ is projective with relatively ample line bundle $O_P(1)$ (this is possible by \autoref{thm:etale_local_proj}),
  \item We have a generating sheaf $\orb E$ on $\orb P$ (this is possible by \autoref{thm:orbinodal_global_quot}).
  \end{compactenum}
  Set $E = \rho_* \orb E$. Since $\orb E \otimes \rho^*O_P(-1)$ is also a generating sheaf, by twisting $\orb E$ by $\rho^*O_P(-1)$ enough times, assume that we have a surjection $O_P^{\oplus M} \to E$ for some $M$.
  
  Let $U \to \st Vect^d_{l,N} \times_{\st M^\orbi}S$ be a surjective map from a scheme (not necessarily of finite type), given by the family  $(\orb P_U \to P_U \to U; \sigma; \orb F)$. It suffices to prove that $(\orb P_U, \orb F)$ is a bounded family of sheaves.
  
  Set \[F = F_{\orb E}(\orb F) = \rho_* \sh Hom_{\orb P}(\orb E, \orb F).\]
  By \autoref{rem:b_for_curves}, it suffices to show that thee degree, rank and $h^0$ of $F_u$ are bounded. The rank of $F_u$ is constant; the degree of $\st Hom(\orb E, \orb F)_u$ is constant. It is easy to see that the degree of $\st Hom(\orb E, \orb F)_u$ and the degree of $F_u$ differ by a bounded amount, depending only on $\orb P \to P$ and $\orb E$. Hence the degree of $F_u$ is bounded. Likewise, it is easy to see that $h^0(F_u)$ and $h^0(\st Hom(\rho_*\orb E, \rho_*\orb F)_u)$ differ by a bounded amount, depending only on $\orb P \to P$ and $\orb E$. On the other hand,
  \begin{align*}
    H^0(\st Hom(\rho_*\orb E, \rho_*\orb F)_u) 
    &= \Hom(E_u, \rho_* \orb F_u) \\
    &\subset \Hom(O_{P_u}^{\oplus M}, \rho_* \orb F_u) \\
    &= H^0(\orb F_u)^{\oplus M}.
  \end{align*}
  By hypothesis, the final vector space has dimension at most $MN$. It follows that $h^0(F_u)$ is bounded. We conclude that $(\orb P_U, \orb F)$ is a bounded family of sheaves.
\end{proof}
We have now finished the second part of the proof of \autoref{thm:big_hurwitz}.\begin{proposition}\label{thm:br_finite_type}
  $\br \from \st H^d \to \st M$ is of finite type.
\end{proposition}
\begin{proof}
  Since the open substacks $\st M^{b,*}$ cover $\st M$, it suffices to show that $\br \from \st H^d_b = \st H^d \times_{\st M} \st M^{b,*} \to \st M^{b,*}$ is of finite type. With $l = -b/2$, and $N$ large enough, we have the following diagram, 
  \[
  \begin{tikzpicture}
    \matrix (m) [matrix of math nodes, row sep=2em, column sep=3em]{
      &{\st H_b^d} && {\st FinCov^d} \\
      &&\st Vect^d_{l,N}&{\st Vect^d}\\
      {\st M^{b,*}}&{\st M^{0,*}}&{\st M^{\orbi\leq N}} & {\st M^\orbi}\\
    };
    \path [->, map]
    (m-1-2) edge [very thick] node[auto]{0} (m-1-4)
    (m-1-4) edge [very thick] node[auto]{1} (m-2-4)
    (m-2-3) edge [very thick] node[auto]{3} (m-3-4)
    (m-3-3) edge [very thick] node[auto]{2} (m-3-2)
    (m-1-2) edge node[auto]{4} (m-3-1)
    (m-1-2) edge (m-2-3)
    (m-1-2) edge (m-3-2)
    (m-1-2) edge (m-3-3)
    (m-2-3) edge (m-2-4)
    (m-3-3) edge (m-3-4)
    (m-2-4) edge (m-3-4)
    (m-3-1) edge (m-3-2)
    ;
  \end{tikzpicture}.
  \]
  The thick arrows in the diagram are known to be of finite type: $(0)$ is an open immersion, $(1)$ is of finite type by \autoref{thm:cover_to_vb}, $(2)$ by \autoref{thm:orbinodal_stack}, and $(3)$ by \autoref{thm:vec_fin_type}. Recall that for algebraic stacks $\orb X$, $\orb Y$, $\orb Z$, all locally of finite type, and morphisms $\orb X \to \orb Y \to \orb Z$, we have the following:
\begin{compactenum}
\item If $\orb X \to \orb Y$ and $\orb Y \to \orb Z$ are of finite type, then $\orb X \to \orb Z$ is also of finite type;
\item If $\orb X \to \orb Z$ is of finite type, then $\orb X \to \orb Y$ is also of finite type.
\end{compactenum}
Using the two repeatedly reveals that $(4)$ is also of finite type.
\end{proof}

\subsection{That $\br \from \st H^d \to \st M$ is Deligne--Mumford}\label{sec:dm}
\begin{proposition}\label{thm:dm}
  $\br \from \st H^d \to \st M$ is representable by Deligne--Mumford stacks.
\end{proposition}
\begin{proof}
  The proof is straightforward. By \autoref{thm:orbinodal_stack}, it suffices to check that $\st H^d \to \st M^\orbi$ is representable by Deligne--Mumford stacks. In other words, we want this morphism to have unramified inertia. This can be checked on points. Let $(\orb P \to P \to \spec k; \sigma; \orb C \to \orb P)$ be a geometric point of $\st H^d$. We must show that $\orb C$ has no infinitesimal automorphisms over the identity of $\orb P$. As $\orb C \to \orb P$ is a finite cover, these automorphisms are classified by $\Hom_{\orb C}(\Omega_{\orb C/ \orb P}, O_{\orb C})$. Since $\orb C \to \orb P$ is unramified on the generic points of the components, $\Omega_{\orb C, \orb P}$ is supported on a zero dimensional locus. Since $\orb C$ is reduced, it follows that $\Hom_{\orb C}(\Omega_{\orb C/ \orb P}, O_{\orb C}) = 0$.
\end{proof}

\subsection{That $\br \from \st H^d \to \st M$ is proper}\label{sec:proper}
We use the valuative criterion. Two pieces of notation will be helpful. If $S$ is the spectrum of a local ring, denote by $S^\circ$ the punctured spectrum 
\[S^\circ =  S  \setminus \{\text{closed point of $S$}\}.\]
For a Deligne--Mumford stack $\orb X$ with coarse space $\orb X \to X$ and a geometric point $x \to X$, set 
\[\orb X_x = \orb X \times_X \spec O_{X,x}.\]
It will be convenient to work with the spectrum of a \emph{henselian} DVR. The reader unfamiliar with this notion should imagine it to be a small (in particular, simply-connected) complex disk.

We begin with a simple lemma about the following setup. Let $r$ be a positive integer and $G$ a finite group. Let $R$ be a henselian DVR with residue field $k$ and uniformizer $t$. Let $O_S$ the henselization of $R[x,y]/(xy-t^r)$ at the point corresponding to $(t,x,y)$. For a positive integer $a$ dividing $r$, define a finite extension $S_a \to S$ by
\[ O_{S_a} = O_S[u,v]/(u^a-x,v^a-y, uv-t^{r/a}).\]
  We have an action of $\mu_a$ on ${S_a}$ over the identity of $S$ by $u \mapsto \zeta u$ and $v \mapsto \zeta^{-1}v$. 
  
\begin{lemma}\label{thm:orbi_cover}
  Let $\chi \from S^\circ \to BG$ be a morphism given by a $G$ torsor $E \to S^\circ$. Then $\chi$ extends to a morphism $[{S_r}/\mu_r] \to BG$. More generally, $\chi$ extends to a morphism  $[{S_a} / \mu_a] \to BG$ if and only if the pullback of $E$ to ${S_a} ^\circ$ is trivial. Furthermore, in this case the extension of $\chi$ is representable if and only if $a$ is the smallest with the above property.
\end{lemma}
\begin{proof}
  To extend $\chi$, we may work \'etale locally on the source. We use the  \'etale cover ${S_a} \to [{S_a}/\mu_a]$. Note that ${S_a}$ is simply connected (it is henselian). Hence the pullback of $E$ to ${S_a}^\circ$ extends to ${S_a}$ if and only if this pullback is trivial. Being trivial over ${S_r}^\circ$ is automatic, since ${S_r}^\circ$ is simply connected.
  
  Note that ${S_r}^\circ \to S^\circ$ is the universal covering space---it is a $\mu_r$-torsor where the source ${S_r}^\circ$ is simply connected. The $G$ torsor $E \to S^\circ$ corresponds to a homomorphism $\mu_r \to G$. By the theory of covering spaces, the pullback of $E$ along ${S_a}^\circ \to S^\circ$ is trivial if and only if $\mu_r \to G$ factors as 
  \begin{equation}\label{eqn:trivial_cover}
    \mu_r \to \mu_a \to G,
  \end{equation}
  where $\mu_r \to \mu_a$ is the map $\zeta \mapsto \zeta^{r/a}$. As we saw, in this case, we get a morphism $\chi \from [{S_a}/\mu_a] \to BG$. Let $s \to [{S_a}/\mu_a]$ be the stacky point. Observe that the map on automorphism groups $\Aut_s([{S_a}/\mu_a]) \to \Aut_s(BG)$ is exactly the map $\mu_a \to G$ in \eqref{eqn:trivial_cover}. Since $\chi$ is representable precisely when $\Aut_s([{S_a}/\mu_a]) \to \Aut_s(BG)$ is injective, the result follows.
\end{proof}

\begin{proposition}\label{thm:separated}
  $\br \from \st H^d \to \st M$ is separated.
\end{proposition}
\begin{proof}
  As $\br$ is of finite type, we may use the valuative criterion. Let $R$ be a henselian DVR with residue field $k$, fraction field $K$ and uniformizer $t$. Set $\Delta = \spec R$. Denote the special, the general and a geometric general point of $\Delta$ by $0$, $\eta$ and $\overline \eta$ respectively. Let $(\orb P_i \to P_i \to \Delta; \sigma; \chi_i \from \orb P_i \to \st A_d)$, for $i = 1, 2$, be two objects of $\st H^d(\Delta)$ over an object $(P;\Sigma;\sigma)$ of $\st M(\Delta)$. Let $\phi_i \from \orb C_i \to \orb P_i$ be the corresponding degree $d$ covers and let
  \[\psi \from (\orb C_1 \to \orb P_1)|_\eta \to \orb (\orb C_2 \to \orb P_2)|_\eta\]
  be an isomorphism over the identity of $P$. We must show that $\psi$ extends to an isomorphism of the orbinodal curves $\orb P_1 \to \orb P_2$ and the covers $\orb C_1 \to \orb C_2$ over all of $\Delta$. Recall that $P^\gen$ is the complement of the markings $\sigma_j$ in the smooth locus of $\orb P \to \Delta$. 
  
  \begin{asparadesc}
  \item[Step 1: Extending $\psi \from \orb C_1 \to \orb C_2$ over $P^{\gen}$:]
    
    Since $\orb C_i \to \orb P$ is \'etale over the generic points of the components of $\orb P|_0$, the map $\psi \from \orb C_1 \to \orb C_2$ extends, except possibly at finitely many points on the central fiber. As a result, on $P^{\gen}$ we get an isomorphism of vector bundles
    \[\psi^\# \from {\phi_2}_*O_{\orb C_2}|_{P^\gen} \to {\phi_1}_*O_{\orb C_1}|_{P^\gen}\] 
    away from a locus of codimension two. Since $P^{\gen}$ is smooth, by Hartog's theorem, this isomorphism extends over all of $P^{\gen}$. The extension must also be an isomorphism of algebras by continuity.
    
  \item[{Step 2: Extending $\psi \from \orb P_1 \to \orb P_2$ at the non-generic nodes}:]
    
    Let $p \to P|_0$ be a node not in the closure of $P|_\eta^\sing$. It suffices to extend $\psi$ \'etale locally around $p$. The local ring $O_{P,p}$ must be the strict henselization of the ring $R[x,y]/(xy-t^r)$ at the point corresponding to $(t,x,y)$ for some positive integer $r$. Recall that the $\chi_i$ are required to map the nodes to the substack $\st E_d \cong B\s_d$ corresponding to \'etale covers. By the first step, the two maps $\chi_i \from \spec O_{P,p}^\circ \to B\s_d$ are isomorphic. Since both $\chi_i$ are representable, the structure of orbinodal curves (\autoref{thm:orbinodal_structure}) and \autoref{thm:orbi_cover} imply that 
    \[ (\orb P_1)_p \cong (\orb P_2)_p \cong \spec [O_{P,p}[u,v]/(u^a-x, v^a-y, uv - t^{r/a})/ \mu_a],\]
    for some divisor $a$ of $r$. Thus, we can get an extension $\psi \from \orb P_1 \to \orb P_2$.

  \item[{Step 3: Extending $\psi \from \orb P_1 \to \orb P_2$ at the marked points}:]
    
    Let $p \to P|_0$ be one of the marked points $\sigma_j(0)$. Then $O_{P,p}$ is the henselization of $R[x]$ at $(t,x)$. Let $\overline{\sigma_j}$ be a geometric generic point of $P$ over $\sigma_j \from \eta \to P|_\eta$. By the structure of orbinodal curves (\autoref{thm:orbinodal_structure}) for $p$ and $\overline{\sigma_j}$, we have the picture for $i = 1, 2$:
    \[
    \begin{tikzpicture}
      \matrix(m)[matrix of math nodes, column sep = 2em, row sep = 1.5em]
      {
        \orb P_{i,p} & \spec O_{P,p}\\
        {[\spec \hens{R[v]} / \mu_{r_i}]} &  \spec \hens{R[x]} \\
        {[\spec \hens{K[v]} / \mu_{r_i}]} &   \spec \hens{K [x]}\\
        \orb P_{i,\overline \sigma_j} & \spec O_{P,\overline\sigma_j}\\
      };
      \path[->] 
      \foreach \i in {1,...,4}
      {
        (m-\i-1) edge (m-\i-2)
      };
      
      \draw[->]
      \foreach \j in {1,...,2}
      {
        (m-1-\j) edge [equal,-] (m-2-\j)
        (m-3-\j) edge (m-2-\j)
        (m-3-\j) edge [equal,-] (m-4-\j)
      };
    \end{tikzpicture},
    \]
    where $\mu_{r_i}$ acts by $v \mapsto \zeta v$. The isomorphism $\orb P_1|_\eta \to \orb P_2|_\eta$ gives an isomorphism $\orb P_{1, \overline\sigma_j} \to \orb P_{2,\overline\sigma_j}$. In particular, we get $r_1 = r_2 = r$. Furthermore, it is easy to see that an isomorphism $[\spec \hens{K[v]}/\mu_r] \to [\spec \hens{K[v]}/\mu_r]$ over the identity of the coarse spaces $\spec \hens{K[x]} \to \spec \hens{K[x]}$ must be of the form $v \mapsto \zeta v$ for some $r$th root of unity $\zeta$. Clearly, such an isomorphism can be extended to an isomorphism $[\spec \hens{R[v]}/\mu_r] \to [\spec \hens{R[v]}/\mu_r]$.
    
  \item[{Step 4: Extending $\psi \from \orb P_1 \to \orb P_2$ at the generic nodes}:]
    
    This step mirrors the previous step; only the orbinodal structures are a little different. We give the details for completeness.

    Let $p \to P|_0$ be a node in the closure of $P|_\eta^\sing$.
    Then $O_{P,p}$ is the henselization of $R[x,y]/xy$ at $(t,x,y)$. Since $\Delta$ is henselian, we have a section $\sigma \from \Delta \to P^{\sing}$ with $\sigma(0) = p$. Let $\overline{\sigma}$ be a geometric generic point of $P|_\eta$ over $\sigma \from \eta \to P|_\eta$. By the structure of orbinodal curves (\autoref{thm:orbinodal_structure}) for $p$ and $\overline \sigma$, we have the picture for $i = 1, 2$:
    \[
    \begin{tikzpicture}
      \matrix(m)[matrix of math nodes, column sep = 2em, row sep = 1.5em]
      {
        \orb P_{i,p} & \spec O_{P,p}\\
        {[\spec \hens{R[u_i,v_i]}/(u_iv_i, u_i-x^{r_i}, v_i-y^{r_i}) / \mu_{r_i}]} &  \spec \hens{(R[x,y]/xy)} \\
        {[\spec \hens{K[u_i,v_i]}/(u_iv_i, u_i-x^{r_i}, v_i-y^{r_i}) / \mu_{r_i}]} &   \spec \hens{(K [x,y]/xy)}\\
        \orb P_{i,\overline \sigma} & \spec O_{P,\overline\sigma}\\
      };
      \path[->] 
      \foreach \i in {1,...,4}
      {
        (m-\i-1) edge (m-\i-2)
      };
      
      \draw[->]
      \foreach \j in {1,...,2}
      {
        (m-1-\j) edge [-, equal] (m-2-\j)
        (m-3-\j) edge (m-2-\j)
        (m-3-\j) edge [-, equal] (m-4-\j)
      };
    \end{tikzpicture}.
    \]
    The isomorphism $\psi \from \orb P_1|_\eta \to \orb P_2|_\eta$ gives an isomorphism $\orb P_{1,\overline\sigma} \to \orb P_{2, \overline\sigma}$. In particular, we get ${r_1} = {r_2} = r$. Furthermore, see that an isomorphism $\psi \from \orb P_{1,\overline\sigma} \to \orb P_{2, \overline\sigma}$ 
    over the identity of the coarse spaces $P_{1,\overline\sigma} \to P_{1,\overline\sigma}$ must be of the form $u_1 \mapsto \zeta_1 u_2$ and $v_1 \mapsto \zeta_2 v_2$ for some $r$th roots of unity $\zeta_1$ and $\zeta_2$. Such an isomorphism can be extended to an isomorphism $\orb P_{1,p}\to \orb P_{2,p}$.
    
  \item[{Step 5: Extending $\psi \from \orb C_1 \to \orb C_2$}:]
    
    By {Step~2, Step~3} and {Step~4} , we have an isomorphism $\psi \from \orb P_1 \to \orb P_2$. By {Step~1}, we also have an isomorphism $\psi \from \orb C_1 \to \orb C_2$ except over the node points and the marked points of $\orb P_i|_0$. However, $\orb C_i \to \orb P_i$ is \'etale over these points; hence $\psi$ must extend to an isomorphism $\psi \from \orb C_1 \to \orb C_2$.
  \end{asparadesc}
\end{proof}

Having proved separatedness, we turn to properness. The crucial ingredient is the following theorem of \citet{horrocks64:_vector}.
\begin{proposition}\label{thm:vb_on_punctured_surface}\citep[Corollary~4.1.1]{horrocks64:_vector}
  Let $S$ be the spectrum of a regular local ring. If $\dim S = 2$, then every vector bundle on the punctured spectrum $S^\circ$ is trivial.
\end{proposition}
\begin{proof}
  We only describe the main idea. See \citep{horrocks64:_vector} for the full details. 

  Denote by $i \from S^\circ \to S$ the inclusion map. Let $E$ be a vector bundle on $S^\circ$. If $\dim S \geq 2$, then $i_*E$ can be shown to be a coherent sheaf on $S$ with depth at least $2$. If $\dim S = 2$, by the Auslander--Buchsbaum formula, we conclude that $i_*E$ is projective, hence free. Therefore, $E$ is free.
\end{proof}

\begin{proposition}\label{thm:proper}
  $\br \from \st H^d \to \st M$ is proper.
\end{proposition}
\begin{proof}
  A large chunk of the proof is identical to the proof in the paper of \citet*[Proposition~6.0.4]{av02}. The final step is new; it uses \autoref{thm:vb_on_punctured_surface} and the expression of $\st A_d$ as the quotient of an affine scheme by $\Gl_d$.

  As $\br$ is of finite type, we may use the valuative criterion. As before, let $R$ be a henselian DVR with residue field $k$, fraction field $K$ and uniformizer $t$. Set $\Delta = \spec R$. Denote the special, the general and a geometric general point of $\Delta$ by $0$, $\eta$ and $\overline \eta$ respectively. Let $(P \to \Delta;\Sigma; \sigma)$ be an object of $\st M(\Delta)$ and $(\orb P|_\eta \to P|_\eta; \sigma; \chi)$ an object of $\st H^d(\eta)$. We want to extend it to an object over all of $\Delta$, possibly after a base change.

  \begin{asparadesc}
  \item[{Step 1. Extending $\chi$ at the generic points of the components}:] 
        
    This step follows {Step 2} in \citep[Proposition~6.0.4]{av02}. We work \'etale locally. Let $\zeta$ be a geometric generic point of a component of $P|_0$. Then the local ring $O_{P,\zeta}$ is also a DVR. Since the branch divisor $\Sigma$ does not contain any component of $P|_0$, the morphism $\chi$ sends the punctured spectrum  $P_\zeta^\circ$ to $\st E_d$. We must extend it to a morphism $\chi \from P_\zeta \to \st E_d$. Since $\st E_d \cong B\s_d$ is a proper Deligne--Mumford stack, such an extension is possible after passing to a finite cover $\tw P_\zeta \to P_\zeta$. By Abhyankar's lemma, there is an $n$ such that $\tw P_\zeta \to P_\zeta$ is isomorphic to $P_\zeta \times_{\spec R} \spec R[\sqrt[n]{t}] \to P_\zeta$. Thus, by passing to a sufficiently big cover $\spec R[\sqrt[N]{t}] \to \spec R = \Delta$, we can extend $\chi$ along the generic points of all the components of $P|_0$. Henceforth, replace $R$ by $R[\sqrt[N]{t}]$. We thus have a morphism $\chi \from P \to \st A_d$ defined away from finitely many points on $P|_0$.
  \item[{Step 2. Extending $\chi$ at the non-generic nodes}:]

    This step follows {Step 3} in \citep[Proposition~6.0.4]{av02}. Let $p \to P|_0$ be a node not in the closure of $P_\eta|^\sing$. We must describe an orbinodal structure at $p$ and a representable extension of $\chi$. It suffices to do both things in the \'etale topology. The stalk $O_{P,p}$ is isomorphic to $\hens{R[x,y]}/(xy-t^r)$ for some $r \geq 1$. Since $\Sigma$ is supported away from the nodes, the morphism $\chi$ sends the punctured spectrum $P_p^\circ$ to $\st E_d \cong B\s_d$. As in \autoref{thm:orbi_cover}, let $a$ be the smallest integer dividing $r$ such that $\chi$ extends to a morphism 
    \[ \chi \from [\spec O_{P,p}[u,v]/(u^a-x, v^a-y, uv-t^{r/a})/\mu_a]\to \st E_d \cong B\s_d ,\]
    where $\mu_a$, as usual, acts by $u \mapsto \zeta$ and $v \mapsto \zeta^{-1}v$. Construct $\orb P$ over $P$ such that
    \[ \orb P_p = [\spec O_{P,p}[u,v]/(u^a-x, v^a-y, uv-t^{r/a})/\mu_a].\]
    By \autoref{thm:orbi_cover}, we have a representable extension $\chi \from \orb P_p \to \st E_d \cong B\s_d.$.
    
  \item[{Step 3: Extending $\chi$ at the generic nodes and marked points}:]
    
    This step follows {Step 4} in \citep[Proposition~6.0.4]{av02}. Let $p \to P|_0$ be in the closure of $P|_\eta^\sing$. First, we extend the orbinode structure $\orb P|_\eta$ over $p$. Note that $O_{P,p}$ is isomorphic to the henselization of $R[x,y]/xy$ at $(t,x,y)$. Since $\Delta$ is henselian, we have a section $\sigma \from \Delta \to P^\sing$ with $\sigma(0) = p$. Letting $\overline\sigma$ be a geometric generic point of this section, we get by \autoref{thm:orbinodal_structure}
    \[ \orb P_{\overline \sigma} \cong [\spec \hens{K[u,v]}/(uv, u^a-x, v^a-y)/\mu_a],\]
    for some positive integer $a$. We extend $\orb P$ over $P_p$ by the same formula
    \[ \orb P_{p} \cong [\spec \hens{R[u,v]}/(uv, u^a-x, v^a-y)/\mu_a].\]
      
      Having defined the orbinodal structure, we extend $\chi$. Again, note that $\chi$ sends a neighborhood of $p$ to the \'etale locus $\st E_d \cong B\s_d$. We work \'etale locally on the source, on the \'etale cover $\spec O_{P,p}[u,v]/(uv, u^a-x,v^a-y) \to \orb P_p$. We already have $\chi$ on the punctured spectrum $(\spec O_{P,p}[u,v]/(uv, u^a-x, v^a-y))^\circ$. Since this punctured spectrum is simply connected, $\chi$ extends to a map $\chi \from O_{P,p}[u,v]/(uv, u^a-x, v^a-y) \to \st E_d$.
      
      The case of marked points $p = \sigma_j(0)$ is entirely analogous, if not easier. 
    \item[{Step 4. Extending $\chi$ over all of $\orb P$}:]

      By the previous steps, we have a pointed orbinodal structure $\orb P \to P$ and an extension of $\chi$ on $\orb P$ away from finitely many smooth, non-stacky points of $\orb P|_0$. Let $p \to P|_0$ be such a point. Recall that $\st A_d \cong [B_d/\Gl_d]$, where $B_d$ is an affine scheme (\autoref{thm:Ad}). The morphism $\chi \from P_p^\circ \to \st A_d$ is equivalent to a $\Gl_d$ torsor $E^* \to P_p^\circ$ and a $\Gl_d$ equivariant morphism $E^* \to B_d$. However, by \autoref{thm:vb_on_punctured_surface}, there are no nontrivial $\Gl_d$ torsors on $P_p^\circ$. In particular, $E^*$ extends to a $\Gl_d$ torsor $E \to P_p$. Next, $E^* \subset E$ is the compliment of the codimension two locus $E|_p$. Since $E$ is smooth and $B_d$ affine, we have an extension $E \to B_d$ by Hartog's theorem. The extension is $\Gl_d$ equivariant by continuity. Thus, we get an extension $\chi \from P_p \to \st A_d$.
    
      Finally, note that the two divisors $\chi^*\Sigma_d$ and $\Sigma$ are supported in the general locus $P^\gen$ and are equal, by construction, on the complement of a codimension two set. Hence, they must be equal. 
    \end{asparadesc}
  \end{proof}
\begin{remark}
  It may be helpful to recast {Step 4} in terms of finite covers. Let $p \to P|_0$ be a smooth point. Assume that we have a finite cover $\phi \from C \to U\setminus\{p\}$, where $U$ is a neighborhood of $p$. We wish to extend it to a cover over all of $U$. By \autoref{thm:vb_on_punctured_surface}, the vector bundle $\phi_*O_C$ extends to a vector bundle over $U$. Next, we must extend the $O_P$ algebra structure of $\phi_*O_C$. The algebra structure is specified by maps of vector bundles, which all extend over $p$ by Hartog's theorem. The extensions continue to satisfy the identities to be an algebra by continuity. We thus get an extension of $\phi$ over all of $U$.
\end{remark}

The proof of the main theorem is now complete. We recall the statement and collect the pieces of the proof.
\thmbighurwitz*
\begin{proof}
  That $\br$ is an algebraic stack, locally of finite type is the content of \autoref{sec:algebraic}, culminating in \autoref{thm:algebraic}. That $\br$ is of finite type is done in \autoref{sec:fin_type}, culminating in \autoref{thm:br_finite_type}. That $\br$ is Deligne--Mumford is \autoref{thm:dm}. Finally, the properness is checked in \autoref{sec:proper} in \autoref{thm:separated} and \autoref{thm:proper}.
\end{proof}

\section{The Local Structure of $\st H^d$}\label{sec:local}
In this section, we analyze the local structure of $\st H^d$. The main consequence  of our analysis  is that $\st H^d$ is smooth for $d = 2$ and $3$ (\autoref{thm:smoothness_for_d23}). Throughout the section, we use the formulation of $\st H^d$ in terms of finite covers instead of in terms of maps to $\st A_d$.

We recall the standard setup of deformation theory. Let $k$ be an algebraically closed field over $\k$. Denote by $\cat{Art}_k$ the category of local Artin rings with residue field $k$. For any object $(A,m)$ of $\cat{Art}_k$, denote by $0$ the special point of $\spec A$. Let $(A,m)$ and $(A',m')$ be two object of $\cat{Art}_k$ related by an exact sequence 
\[ 0 \to J \to A' \to A \to 0.\]
Say that $A'$ is a \emph{small extension} of $A$ by $J$ if $m'\cdot J = 0$. Denote by $\Def_X$ the standard functor on $\cat{Art}_k$ classifying deformations of $X$, namely
\[ \Def_X(A) = \{(X_A \to \spec A, i)\},\]
where $X_A \to \spec A$ is a flat morphism and $i \from X_A|_0 \to X$ an isomorphism. We shorten $(X_A \to \spec A, i)$ to just $X_A$, and call it a \emph{deformation of $X$ over $A$}. Likewise, for a morphism $\phi \from X \to Y$, denote by $\Def_\phi$ the functor classifying deformations of $\phi$ (allowing both $X$ and $Y$ to vary), namely
\[ \Def_\phi(A) = \{(X_A \to \spec A, Y_A \to \spec A ,\phi_A \from  X_A \to Y_A, i_X, i_Y)\},\]
where $X_A \to \spec A$ and $Y_A \to \spec A$ are flat morphisms and $i_X \from X_A|_0 \to X$ and $i_Y \from Y_A|_0 \to Y$ are isomorphisms making the obvious commutative diagram
\begin{equation}\label{eqn:commute}
\begin{tikzpicture}
  \matrix (m) [matrix of math nodes, row sep=2em, column sep=2em]{
    X_A|_0 & Y_A|_0 \\
    X & Y\\
  };
  \path [->, math, map]
  (m-1-1) edge node [auto] {\phi_A|_0} (m-1-2)
  (m-2-1) edge node [auto] {\phi} (m-2-2)
  (m-1-1) edge node [auto] {i_X} (m-2-1)
  (m-1-2) edge node [auto] {i_Y} (m-2-2);
\end{tikzpicture}.
\end{equation}
We shorten the unwieldy $(X_A \to \spec A, Y_A \to \spec A ,{\phi_A} \from X_A \to Y_A, i_X, i_Y)$ to just $(\phi_A \from X_A \to Y_A)$ and call it a \emph{deformation of $\phi$ over $A$}. 

Let $\xi = (\orb P \to P; \sigma_1, \dots, \sigma_n; \phi \from \orb C \to \orb P)$ be such that $(\orb P \to P; \sigma_1, \dots, \sigma_n)$ is a (not necessarily proper) pointed orbinodal curve over $k$ and $\phi \from \orb C \to \orb P$ a finite cover, \'etale over the nodes and the marked points of $P$. Denote by $\Def_\xi$ the functor classifying deformations of $\xi$:
\[ \Def_\xi(A) = \{(\orb P_A \to P_A \to \spec A; \sigma_{i,A}; \phi_A \from \orb C_A \to \orb P_A, i_C, i_P)\},\]
where $(\orb P_A \to P_A\to\spec A; \sigma_{i,A})$ is a (not necessarily proper) pointed orbinodal curve, $\phi \from \orb C_A \to \orb P_A$ a finite cover, and $i_P \from \orb P_A|_0 \to \orb P$ and $i_C \from \orb C_A|_0 \to \orb C$ isomorphisms commuting with $\phi_A$ and $\phi$ as in \eqref{eqn:commute}. If $\xi$ corresponds to a point of $\st H^d$, then we have a formally smooth morphism $\Def_{\xi} \to \st H^d.$ Our goal is to understand $\Def_{\xi}$. 

Following \citet*[\S~4.1]{fedorchuk10:_modul_hyp}, we first simplify the task of studying the deformations of $\xi$ into the study of its deformations on Zariski local pieces. Following his terminology \citep[\S~4.1]{fedorchuk10:_modul_hyp}, let $\{U_i\}$ be an adapted affine open cover of $P$. This means that each $U_i$ contains exactly one from the following: a node, a marked point or a point of $\supp (\br\phi)$. Set
\begin{align*}
  \orb U_i &= U_i \times_P \orb P \\
  \orb V_i &= \orb C \times_{\orb P}\orb U_i\\
  \phi_i &= \phi|_{\orb V_i} \from \orb V_i \to \orb U_i\\
  \xi_i &= (\orb U_i \to U_i; \sigma_i; \phi_i \from \orb V_i \to \orb U_i).
\end{align*}
In the last equation, $\sigma_i$ is ignored if $U_i$ does not contain any marked point. Set $\orb U_{ij} = \orb U_i \cap \orb U_j$, $\orb V_{ij} = \orb V_i \cap \orb V_j$, $\orb U_{ijk} = \orb U_i \cap \orb U_j \cap \orb U_k$, and so on. Observe that $\orb U_{ij}$ does not contain orbinodes, marked points or branch points. To emphasize that these multiple intersections are schemes, we denote them by roman letters $U_{ij}$, $V_{ij}$, $U_{ijk}$, and so on.

We have restriction maps $\Def_\xi \to \Def_{\xi_i}$.
\begin{proposition}\label{thm:def_global_to_local}
  With the above notation, the map $\Def_\xi \to \prod_i \Def_{\xi_i}$ is formally smooth.
\end{proposition}
\begin{proof}
  Let $0 \to k \to A' \to A \to 0$ be a small extension. Assume that we are given a deformation $\xi_A$ of $\xi$ on $A$. Denote the restriction of $\xi_A$ to $\orb U_i$ by $\xi_{i,A}$; it is a deformation of $\xi_i$. Suppose, furthermore, that we are given extensions $\xi_{i,A'}$ of $\xi_{i,A}$. We must prove that the $\xi_{i,A'}$ can be glued to get a global extension $\xi_{A'}$ of $\xi_{A}$.

  Note that, by construction, $U_{ij}$ is a nonsingular affine scheme. Therefore, its deformations are trivial. Let $p_{ij} \from O_{\orb U_{i,A'}}|_{U_{ij}} \to O_{\orb U_{j,A'}}|_{U_{ij}}$
  be an isomorphism over the identity
  \[ O_{\orb P_A}|_{U_{ij}} = O_{\orb U_{i,A}}|_{U_{ij}} \to O_{\orb U_{j,A}}|_{U_{ij}} = O_{\orb P_A}|_{U_{ij}}.\]
  The choice of $p_{ij}$ is given by an element of $\Hom (\Omega_{U_{ij}}, O_{U_{ij}})$. The isomorphisms $p_{ij}$ may not be compatible on the triple overlaps $U_{ijk}$. However, since $H^2(\sh Hom(\Omega_{\orb P}, O_{\orb P})) = 0$,
  the two co-cycle defined by $p_{ij} + p_{jk} - p_{ik}$ on $U_{ijk}$ is in fact a co-boundary. As a result, by changing the choice of the $p_{ij}$, we can assure that they are compatible on triple overlaps. Thus, we obtain an orbinodal curve $(\orb P_{A'} \to P_{A'}; \sigma_{A'})$ over $A'$ extending $(\orb P_A \to P_A; \sigma_A)$ over $A$. This takes care of one piece of an extension $\xi_{A'}$ of $\xi_A$.

  Having constructed $\orb P_{A'}$, we construct $\orb C_{A'}$ similarly by choosing isomorphisms 
  \[c_{ij} \from O_{\thk V_{i,A'}}|_{V_{ij}} \to O_{\thk V_{j,A'}}|_{V_{ij}}.\]
  Since $\phi \from V_{ij} \to U_{ij}$ is \'etale, we have an equality $\phi^*\Omega_{U_{ij}} = \Omega_{V_{ij}}$. Observe that if we wish to extend $\phi_A \from \orb C_A \to \orb P_A$ to $\phi_{A'} \from \orb C_{A'} \to \orb P_{A'}$, where $\orb P_{A'}$ is glued by the $p_{ij}$ and $\orb C_{A'}$ by the $c_{ij}$, then $c_{ij} \in \Hom(\Omega_{V_{ij}}, O_{V_{ij}})$ must be the pullback of $p_{ij} \in \Hom(\Omega_{U_{ij}}, O_{U_{ij}})$. By choosing the $c_{ij}$ in this way, we get the desired extension $\orb C_{A'}$ of $\orb C_A$ along with an extension $\phi_{A'} \from \orb C_{A'} \to \orb P_{A'}$ of $\phi_A \from \orb C_A \to \orb P_A$, completing the second piece of the extension $\xi_{A'}$ of $\xi_A$.
\end{proof}

Next, we analyze $\Def_{\xi_i}$. We use the forgetful morphisms $\Def_{\xi_i} \to \Def_{\orb U_i}$ and $\Def_{\xi_i} \to \Def_{\orb V_i}$.
\begin{proposition}\label{thm:def_local}
  Retain the notation of \autoref{thm:def_global_to_local}.
  \begin{compactenum}
  \item If $\orb U_i$ does not contain a point of $\br\phi$, then $\Def_{\xi_i}$ is formally smooth.
  \item If $\orb U_i$ contains a point of $\br\phi$, then $\Def_{\xi_i} \to \Def_{\orb V_i}$ is formally smooth.
  \end{compactenum}
\end{proposition}
\begin{remark}
  In the second case, $\orb U_i$ does not contain any orbinode or marked point. Hence, it is a nonsingular scheme and $\Def_{\xi_i}$ is simply $\Def_{\phi_i}$.
\end{remark}
\begin{proof}
  In the first case, the map $\phi_i \from \orb V_i \to \orb U_i$ is \'etale.   Therefore, the forgetful map $\Def_{\xi_i} \to \Def_{(\orb U_i; \sigma_i)}$ is an isomorphism. We are thus reduced to showing that the deformations of the pointed orbinodal curve $(\orb U_i; \sigma_i)$ are unobstructed. This is shown in \citep[\S~3]{acv:03}. We briefly recall the argument. The obstructions to the deformations lie in $\sh Ext^2(\Omega_{\orb U_i}, O_{U_i})$. \'Etale locally, $\orb U_i$ is at worst a nodal curve; hence $\sh Ext^2(\Omega_{\orb U_i}, O_{U_i}) = 0$.

  In the second case, $\orb U_i = U_i$ is a nonsingular affine scheme; its deformations are trivial. For the smoothness of $\Def_{\phi_i} \to \Def_{V_i}$, take an extension $A' \to A \to 0$ of rings in $\cat{Art}_k$, a deformation $\phi_{i,A} \from V_{i,A} \to U_i \times \spec A$ of $\phi_i$ over $A$ and an extension $V_{i,A'} \to \spec A'$ of $V_{i,A}$. We must construct an extension $\phi_{i,A'} \from \thk V_{i,A'} \to U_i \times \spec A'$ of $\phi_{i,A}$. By the infinitesimal lifting property for $U_i$, the map $V_{i,A} \to U_i$ extends to a map $V_{i, A'} \to U_i$, yielding such an extension $\phi_{i,A'} \from \thk V_{i,A'} \to U_i \times \spec A'$.
\end{proof}

 Recall that a scheme (stack) is \emph{smoothable} if it is the flat limit of non-singular schemes (stacks). Let $\orb  H^d \subset \st H^d$ be the open locus consisting of 
\[ (\orb P \to P; \sigma; \phi \from \orb C \to \orb P), \]
where $\orb C$ and $\orb P$ are smooth and $\phi$ is simply branched. 
\begin{proposition}\label{thm:locally_smooth_implies_smooth}
  Retain the notation of \autoref{thm:def_global_to_local}. Let $S$ be the set of indices $i$ for which $U_i$ contains a point of $\br\phi$.
  \begin{compactenum}
  \item $\Def_{\xi}$ is smooth if and only if $\Def_{V_i}$ is smooth for all $i \in S$. 
  \item The point of $\st H^d$ given by $\xi$ is in the  closure of $\orb H^d$ if and only if $V_i$ is smoothable for all $i \in S$.
  \end{compactenum}
\end{proposition}
\begin{proof}
  \autoref{thm:def_global_to_local} and \autoref{thm:def_local} together give a smooth morphism $\Def_\xi \to \prod_{i\in S} \Def_{V_i}$, proving the first assertion.  For the second, consider the smooth morphism
  \begin{equation}\label{eqn:smooth_to_orbinodes_and_singularities}
    \Def_\xi \to \prod_{i \not \in S} \Def_{\orb U_i} \times \prod_{i \in S} \Def_{V_i}. 
\end{equation}
For $i \not \in S$, the $\orb U_i$ is either a smooth curve or an orbinodal curve. In either case, it is smoothable. By the smoothness of \eqref{eqn:smooth_to_orbinodes_and_singularities}, if all the $V_i$ are smoothable for $i \in S$ then $\xi$ is in the closure of the locus of 
    \[ (\orb P \to P; \sigma; \phi \from \orb C \to \orb P), \]
    with smooth $\orb C$ and $\orb P$. It is not hard to see that this locus is in the closure of $\orb H^d$, where the only additional constraint is that $\phi$ be simply branched.
\end{proof}

We record two important special cases.
\begin{theorem}\label{thm:smoothness_for_d23}
  For $d = 2$ and $3$, the stack $\st H^d$ is smooth and contains $\orb H^d$ as a dense open substack.
\end{theorem}
\begin{proof}
  We begin with a general observation. For a finite cover $\phi \from X \to Y$ of degree $d$, we have an exact sequence
  \[ 0 \to O_Y \to \phi_* O_X \to F \to 0,\]
  split by $1/d$ times the trace map $\tr \from \phi_* O_X \to O_Y$. Therefore, the vector bundle $F$ admits a map $F \to \phi_*O_X$. Since $\phi_*O_X$ is a sheaf of $O_Y$ algebras, we get a  map $\Sym^*(F) \to \phi_*O_X$, which is clearly surjective. In other words, $\phi \from X \to Y$ naturally factors as an embedding 
  \begin{equation}\label{eqn:canonical_scroll_embedding}
    \iota \from X \into \spec_Y\Sym^*(F)
  \end{equation}
  followed by the projection $\spec_Y\Sym^*(F) \to Y$.

  We now prove the theorem. By \autoref{thm:locally_smooth_implies_smooth}, it suffices to prove that $\Def_{V_i}$ is smooth and $V_i$ is smoothable for all $i$ for which $\phi_i \from V_i \to U_i$ is ramified. In the case of $d = 2$, the embedding $\iota$ in \eqref{eqn:canonical_scroll_embedding} exhibits $V_i$ as a divisor in a nonsingular affine surface. It is now well-known that $\Def_{V_i}$ is smooth and $V_i$ is smoothable. In the case of $d = 3$, the embedding $\iota$ exhibits $V_i$ as a subscheme of a nonsingular affine threefold. Since $V_i$ is a reduced curve, it is Cohen--Macaulay. Thus $V_i$ is a Cohen--Macaulay subscheme of codimension two in a nonsingular affine variety. This lets us conclude that $\Def_{V_i}$ is smooth \citep[\S~2.8]{hartshorne10:_defor} and $V_i$ is smoothable \citep[Theorem~2]{schaps77:_defor_cohen_macaul_schem_codim}.
\end{proof}

\section{Projectivity}\label{sec:projectivity}
In this section, we prove that the branch morphism is projective on coarse spaces by showing that the Hodge line bundle is relatively anti-ample. We begin by defining the Hodge bundle.

Let $(\orb P \to P; \sigma; \phi \from \orb C \to \orb P)$ be the universal object over $\st H^d$. Let $\pi_{\orb P} \from \orb P \to \st H^d$ and $\pi_{\orb C} \from \orb C \to \orb H^d$ be the projections. When no confusion is likely, we denote both by $\pi$. Define the \emph{Hodge bundle} $\Lambda$ on $\st H^d$ by
\[ \Lambda = \dual{(R^1\pi_* O_{\orb C})}.\]
Then $\Lambda$ is a locally free sheaf on $\st H^d$. Define the line bundle $\lambda$ by
\[ \lambda = \det \Lambda.\]
We use additive notation for $\lambda$. So, $-\lambda$ denotes the dual of $\lambda$.

Throughout, we use without explicit reference that separated Deligne--Mumford stacks have coarse spaces \citep[Corollary~1.3]{keel97:_quotien}. We also repeatedly use that Deligne--Mumford stacks admit a finite surjective map from a scheme \citep[Proposition~2.6]{vistoli89:_inter}. This is typically used in the following guise: if we have a map from $X$ to the coarse space $Y$ of a Deligne--Mumford stack $\orb Y$, then there is $\tw X \to X$, finite and surjective, such that $\tw X \to Y$  lifts to $\tw X \to \orb Y$.

\begin{theorem}\label{thm:projectivity}
  Let $\orb M$ be a Deligne--Mumford stack separated over $\k$ and let $\orb M \to \st M$ be a morphism. Set $\orb H = \orb M \times_{\st M} \st H^d$. Denote by $H$ and $M$ the coarse spaces of $\orb H$ and $\orb M$ respectively. Then the induced morphism 
  \[ \br \from H \to M\]
  is projective. In particular, if $M$ is projective, so is $H$.
\end{theorem}
The essential ingredient in the proof is the following lemma.
\begin{lemma}\label{thm:lambda_ample_on_fibers}
  Let $s \from \spec k \to \st M$ be a geometric point, and $X$ a scheme with a quasi-finite morphism $X \to s \times_{\st M} \st H^d$. Then the pullback  of $-\lambda$ to $X$ is ample.
\end{lemma}
\begin{proof}
  Without loss of generality, $X$ is reduced and connected. By replacing $X$ by its normalization $X^\nu \to X$ if necessary, assume further that $X$ is normal. Let $(P; \Sigma; \sigma)$ be the marked nodal curve over $k$ corresponding to the point $s$ and $(\orb P \to P \times X; \sigma \times X; \phi \from \orb C \to \orb P)$ the family over $X$ giving the map to $s \times_{\st M} \st H^d$.  Construct  $\tw{\orb C} \to \orb C$ by normalizing $\orb C$ over ${\orb P}^\sm$. Explicitly, $\tw{\orb C}$ is such that we have
  \begin{align*}
    \tw{\orb C} \times_{P} (P\setminus\Sigma)  &= \orb C \times_{P} (P \setminus \Sigma), \text{ and }\\
    \tw{\orb C} \times_{P} {\orb P}^\sm &= (\orb C \times_{P} {\orb P}^\sm)^\nu.
  \end{align*}
  It is easy to see using the result of \citet{teissier80:_resol_i} that the fibers of $\tw {\orb C} \times_{P} P^\sm \to X$ are the normalizations of the corresponding fibers of $\orb C \times_{P} P^\sm \to X$.

 Consider the family of finite covers $\tw \phi \from \tw{\orb C} \to \orb P$ over $X$. Let $t \to X$ be a $k$-point. Then $\tw{\orb C_t}$ is smooth except over the nodes of $\orb P_t$ and \'etale over the nodes of $\orb P_t$. This implies that there are only finitely many isomorphism types for the cover $\tw{\orb C_t} \to \orb P_t$. Since $X$ is connected, the fibers over $X$ of $\tw \phi \from \tw{\orb C} \to \orb P$ must all be isomorphic as finite covers. By replacing $X$ by a finite cover if necessary, we can make $\tw \phi \from \tw{\orb C} \to \orb P$ a constant family. In other words, we get $\tw{\phi_0} \from \orb C_0 \to \orb P_0$ over $k$ such that
  \[
  \tw{\orb C} = \tw{\orb C_0} \times X, \quad
  \orb P = {\orb P_0} \times X, \text{ and }
  \tw \phi  = \tw{\phi_0} \times X.
  \]

    In the rest of the proof, we treat $O_{\orb C}$ and $O_{\tw {\orb C}}$ as bundles on $\orb P$, omitting $\phi_*$ and $\tw \phi_*$ to lighten notation. Denote by $I_{\Sigma}$ the ideal of $\Sigma$ in $\orb P$. The inclusion $O_{\orb C} \subset O_{\tw {\orb C}}$  is an isomorphism except over $\Sigma \times X$. Hence, the quotient $O_{\tw {\orb C}}/O_{\orb C}$ is annihilated by $I_{\Sigma_0 \times X}^N$ for $N$ large enough. In other words, for every point $t$ of $X$, we have
    \begin{equation}\label{eqn:conductor}
      I_{\Sigma}^N \cdot O_{\tw{\orb C_t}} \subset O_{{\orb C_t}}.
    \end{equation}
    As a result, $O_{\orb C_t}$ is determined by the subspace $H^0(O_{\orb C_t}/I_{\Sigma}^N\cdot O_{\tw{\orb C_t}})$ of $H^0(O_{\tw{\orb C_t}}/I_{\Sigma}^N\cdot O_{\tw{\orb C_t}})$.

    Consider the following sequence on $\orb P$:
    \[ 0 \to O_{\orb C}/(I_{\Sigma \times X}^N \cdot O_{\tw {\orb C}}) \to O_{\tw{\orb C}} / (I_{\Sigma \times X}^N \cdot O_{\tw {\orb C}}) \to O_{\tw {\orb C}}/ O_{\orb C} \to 0.\]
    Applying $\pi_*$, we obtain a sequence of vector bundles on $X$:
    \begin{equation}\label{eqn:push_conductor}
      0 \to \pi_*\left(O_{\orb C}/(I_{\Sigma \times X}^N \cdot O_{\tw {\orb C}}) \right) \to \pi_* \left(O_{\tw{\orb C}} / (I_{\Sigma \times X}^N \cdot O_{\tw {\orb C}})\right) \to \pi_* \left(O_{\tw{\orb C}}/O_{\orb C}\right) \to 0.
    \end{equation}
    Since $\tw {\orb C} =  \tw {\orb C_0} \times X$, the middle vector bundle is in fact trivial:
\[ \pi_*\left(O_{\tw{\orb C}}/(I_{\Sigma_0 \times X}^N \cdot O_{\tw{\orb C}})\right) = V \otimes O_X, \text{ where }
V = H^0\left(O_{\tw{\orb C_0}}/(I_{\Sigma_0}^N \cdot O_{\tw{\orb C_0}})\right).\]
The sequence \eqref{eqn:push_conductor} gives us a morphism $\mu \from X \to \G$, where $\G$ is the Grassmannian of quotients of $V$ of the appropriate dimension. Moreover, by our discussion above, for every point $t$ of $X$, the fiber $\phi_t \from \orb C_t \to \orb P_t$ is determined by $\mu(t)$. Since $X \to s \times_{\st M}\st H^d$ is quasi-finite, $\mu$ must also be quasi-finite. We conclude that the pullback to $X$ of the Pl\"ucker line bundle on $\G$ is ample. By \eqref{eqn:push_conductor}, this pullback is simply $\det\pi_*\left(O_{\tw{\orb C}}/O_{\orb C}\right)$.
On the other hand, applying $\pi_*$ to the exact sequence
\[ 0 \to O_{\orb C} \to O_{\tw{\orb C}} \to O_{\tw{\orb C}}/O_{\orb C} \to 0,\]
and keeping in mind that $\tw{\orb C} = \tw{\orb C_0} \times X$ is a constant family, we get
\[ \det \pi_* \left(O_{\tw {\orb C}}/O_{\orb C}\right) \cong \det R^1\pi_* O_{\orb C}.\]
We deduce that the right hand side, which is the pullback of $-\lambda$ to $X$, is ample.
\end{proof}
\begin{proof}[Proof of \autoref{thm:projectivity}]
  We want to show that $\br \from H \to M$ is projective. Denote also by $\lambda$ the pullback to $\orb H$ of $\lambda$ on $\st H^d$. Since $\Pic(\orb H) \otimes \Q = \Pic(H) \otimes \Q$, we may treat $\lambda$ as a $\Q$ line bundle on $H$. We claim that $-\lambda$ is $\br$-ample. It suffices to check this on the fibers of $\br \from H \to M$. Let $s \to M$ be a $k$-point and set $H_s = \br^{-1}(s)$. Choose a lift $\overline s \to \orb M$ of $s \to M$. Then $H_s$ is the coarse space of $\overline s \times_{\orb M} \orb H$. There is a scheme $X$ and a finite surjective map $X \to \o s \times_{\orb M} \orb H$. \autoref{thm:lambda_ample_on_fibers} implies that $-\lambda$ is ample on $X$. Since $X \to H_s$ is finite and surjective, we deduce that $-\lambda$ is ample on $H_s$.
\end{proof}

\subsection{Spaces of weighted admissible covers}\label{sec:projective_compactifications}
The proper morphism $\st H \to \st M$ lets us construct several compactifications of different variants of the Hurwitz spaces. Some of these have appeared in literature  in different guises. Fix  non-negative integers $g$, $h$, and $b$ related by
\[ 2g-2 = d(2h-2)+b.\]

Let $\st M_{h;b} \subset \st M$ be the open and closed substack whose $k$ points correspond to $(P; \Sigma)$, where $P$ is a connected curve of arithmetic genus $h$ and $\Sigma \subset P$ a divisor of degree $b$. Let $\orb M_{h;b} \subset \st M_{h;b}$ be the open substack where $P$ is smooth and $\Sigma$ is reduced. Then $\st M_{h;b}$ is a smooth stack and it contains $\orb M_{h;b}$ as a dense open substack.

Let $\st H_{g/h}^d \subset \st M_{h;b} \times_{\st M} \st H^d$ be the open and closed substack whose $k$ points correspond to $(\orb P \to P; \phi \from \orb C \to \orb P)$, where $\orb C$ is connected. By the Riemann--Hurwitz formula, $\orb C$ has (arithmetic) genus $g$. Observe that the small Hurwitz stack $\orb H_{g/h}^d$ is simply the open substack defined by
\[ \orb H_{g/h}^d =  \orb M_{h;b} \times_{\st M_{h;b}} \st H_{g/h}^d.\]
It parametrizes $(\phi \from C \to P)$ where $P$ and $C$ are smooth and $\phi$ is simply branched.

 We recall a sequence of open substacks of $\st M_{h;b}$ that contain $\orb M_{h;b}$ and are proper over the base field. These are the spaces of \emph{weighted pointed stable curves} constructed by \citet*{hassett03:_modul}.

\begin{definition}
  Let $\epsilon$ be a rational number. Let $P$ be a nodal curve over $k$ and $\Sigma \subset P$ a divisor supported in the smooth locus. We say that $(P,\Sigma)$ is \emph{$\epsilon$-stable} if
  \begin{compactenum}
  \item for every point $p$ of $P$, we have
    \[ \epsilon \cdot \mult_p(\Sigma) \leq 1;\]
  \item the $\Q$ line bundle $\omega_P \otimes O_P(\epsilon\Sigma)$ is ample, where $\omega_P$ is the dualizing line bundle of $P$.
  \end{compactenum}
  Denote by $\o{\orb M}_{h;b}(\epsilon) \subset \st M_{h;b}$ the open substack parametrizing $\epsilon$-stable marked curves. 
\end{definition}

Recall the main theorem from \citep{hassett03:_modul}.
\begin{theorem}\label{thm:hassett_spaces}\citep[Theorem~2.1, Variation~2.1.3]{hassett03:_modul}
  $\o{\orb M}_{h;b}(\epsilon)$ is a Deligne--Mumford stack, proper over $\k$. It admits a projective coarse space $\o{\sp M}_{h;b}(\epsilon)$.
\end{theorem}
If $\deg(\omega_P(\epsilon\Sigma)) = \epsilon \cdot b + 2h-2 \leq 0$, then $\o{\orb M}_{h;b}(\epsilon)$ is empty. Otherwise, it contains $\orb M_{h;b}$ as a dense open substack.

\begin{definition}
  Define the stack $\o{\orb H}_{g/h}^d(\epsilon)$ of \emph{$\epsilon$-admissible covers} by the formula
  \[ \o{\orb H}_{g/h}^d(\epsilon) = \o{\orb M}_{h;b}(\epsilon) \times_{\st M_{h;b}} \st H_{g/h}^d.\]
  We sometimes call $\epsilon$-admissible covers \emph{weighted admissible covers}.
\end{definition}

\begin{corollary}\label{thm:e_admissible_covers}
  $\o{\orb H}_{g/h}^d(\epsilon)$ is a Deligne--Mumford stack, proper over $\k$. It admits a projective coarse space $\o{\sp H}_{g/h}^d(\epsilon)$ and a morphism
  \[ \br \from \o{\orb H}_{g/h}^d(\epsilon) \to \o{\orb M}_{h;b}(\epsilon).\]
\end{corollary}
\begin{proof}
  Follows directly from \autoref{thm:big_hurwitz} and \autoref{thm:projectivity}.
\end{proof}
As before, if $\epsilon\cdot b + 2h-2 \leq 0$, then $\o {\orb H}_{g/h}^d(\epsilon)$ is empty. Otherwise, it contains ${\orb H}_{g/h}^d$ as an open substack (but it may not be dense; see \autoref{ex:extra_components}).

\subsection{Examples}
We describe the geometry of the spaces of weighted admissible covers by some illustrative examples.

These spaces generalize some known compactifications of Hurwitz spaces, mentioned in the following two examples.
\begin{example}[Twisted admissible covers]
  Consider the case $\epsilon = 1$ and the resulting stack of $1$-admissible covers $\o{\orb H}^d_{g/h}(1)$. It parametrizes $(\orb P \to P; \phi \from \orb C \to \orb P)$, where $\br\phi \subset P$ is \'etale over the base. The induced morphism on coarse spaces $C \to P$ is an admissible cover in the sense of \citet*{Harris82:_Kodair_Dimen_Of_Modul_Space_Of_Curves} (but with unordered branch points).

  By \autoref{thm:locally_smooth_implies_smooth}, the stack $\o{\orb H}^d_{g/h}(1)$ is smooth and contains the small Hurwitz stack $\orb H_{g/h}$ as a dense open substack. In fact, $\o{\orb H}^d_{g/h}(1)$ is essentially the stack of \emph{twisted admissible covers} of \citet*{acv:03}; the only difference is that in \citep{acv:03}, the branch points are ordered, whereas in $\o{\orb H}^d_{g/h}(1)$, they are unordered.
\end{example} 

\begin{example}[Spaces of hyperelliptic curves]\label{ex:hyperelliptic}
  Consider the case $h = 0$ and $d = 2$, and the resulting stacks $\o{\orb H}^2_g(\epsilon)$ of $\epsilon$-admissible covers. Consider a  $k$-point of $\o{\orb H}^2_g(\epsilon)$, given by a cover $(\orb P \to P; \phi \from \orb C \to \orb P)$. Say $\lfloor 1/\epsilon \rfloor = n$. Away from over the nodes of $\orb P$, the singularities of $\orb C$ are (\'etale) locally of the form 
  \[ y^2-x^m,\]
  for $m \leq n$. Thus, the spaces $\o{\sp H}^2_g(\epsilon)$ are just the spaces of hyperelliptic curves with $A_n$ singularities constructed by \citet{fedorchuk10:_modul_hyp}.
\end{example}

The singularities of $\orb C$ get much more interesting for higher degrees, as illustrated in the next example.
\begin{example}[Singularities of $\orb C$]\label{ex:singularities}
  Let $(\orb P \to P; \phi \from \orb C \to \orb P)$ be a $k$-point of $\o{\orb H}_{g/h}^d(\epsilon)$. Notice that we do not explicitly restrict the singularities of $\orb C$; the restrictions are imposed indirectly by the allowed multiplicity of the branch divisor. We list some examples of the singularities that appear on $\orb C$ for small values of $1/\epsilon$ and $d \geq 3$.
  \begin{compactenum}
  \item \mbox{$1/2 < \epsilon \leq 1$}
    
    In this case, $\orb C$ is smooth (except, of course, over the nodes of $\orb P$) and simply branched over $\orb P$. 
    
  \item \mbox{$1/3 < \epsilon \leq 1/2$}

    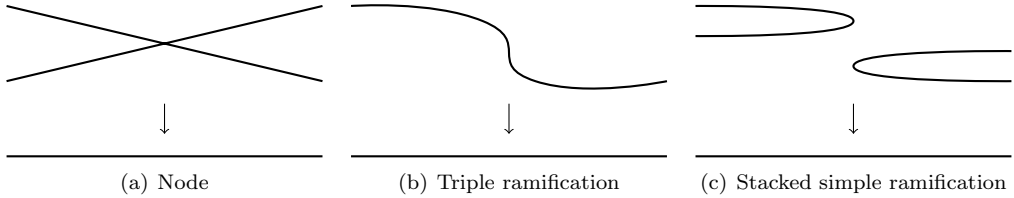
\begin{figure}[t]
      \centering
      \subfigure[Node]{
        \begin{tikzpicture}[yscale=.2, xscale=.7, thick]
          \draw [smooth, tension=1]
          plot coordinates{(3,4) (-3, -1)} 
          plot coordinates{(3,-1) (-3,4)};
          \draw (-3,-6) -- (3, -6);
          \path (0,-2.5) edge [->, thin] (0,-4.5);
        \end{tikzpicture}
        \label{fig:node}
      }
      \subfigure[Triple ramification]{
        \begin{tikzpicture}[yscale=.2, xscale=.7, thick]
          \draw [smooth, tension=1]
          plot coordinates{(-3,4) (-.5, 3) (.5,-1) (3,-1)};
          \draw (-3,-6) -- (3, -6);
          \path (0,-2.5) edge [->, thin] (0,-4.5);
        \end{tikzpicture}
        \label{fig:triple}
      }
      \subfigure[Stacked simple ramification]{
        \begin{tikzpicture}[yscale=.2, xscale=.7, thick]
          \draw [smooth, tension=2]
          plot coordinates{(-3,4) (0, 3) (-3, 2)}
          plot coordinates{(3,1) (0, 0) (3, -1)};
          \draw (-3,-6) -- (3, -6);
          \path (0,-2.5) edge [->, thin] (0,-4.5);
        \end{tikzpicture}
        \label{fig:stacked}
      }
      
      \caption{Possible local pictures of $\phi$ for $1/3 < \epsilon \leq 1/2$}
    \end{figure}

    In this case, $\orb C$ can have only nodal singularities. Also, the branches of the nodes must be individually unramified over $\orb P$ as in \autoref{fig:node}. This case also allows certain kinds of multiple ramification in $\phi$: it can be triply ramified as in \autoref{fig:triple} or it can have two simple ramification points lying over the same point of $\orb P$ as in \autoref{fig:stacked}.

  \item \mbox{$1/4 < \epsilon \leq 1/3$}

    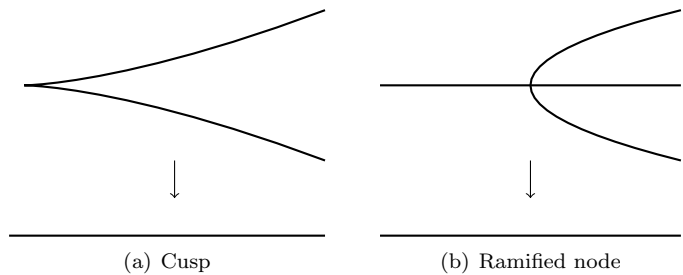
\begin{figure}[t]
      \centering
      \subfigure[Cusp]{
        \begin{tikzpicture}[domain=-1:1, thick, math, xscale=2]
          \draw plot[parametric, id=cusp] function{2*t*t,t*t*t};
          \draw (-.1,-2) -- (2,-2);
          \path (1,-1) edge [->, thin] (1,-1.5);
        \end{tikzpicture}
        \label{fig:cusp}
      }
      \quad
      \subfigure[Ramified node]{
        \begin{tikzpicture}[domain=-1:1, thick, math, xscale=2]
          \draw plot[parametric,id=ramnode1] function{t*t, t};
          \draw plot[parametric,id=ramnode2] function{t, 0};
          \draw (-1,-2) -- (1,-2);
          \path (0,-1) edge [->, thin] (0,-1.5);
        \end{tikzpicture}
        \label{fig:ramified_node}
      }
      \caption{Some of the possible local pictures of $\phi$ for $1/4 < \epsilon \leq 1/3$}
    \end{figure}
    
    In this case, $\orb C$ can have nodal and cuspidal (formally $k\f{x,y}/(y^2-x^3)$) singularities as in \autoref{fig:cusp}. This case also allows even more multiple ramification in $\phi$; for example, it is possible to have ramification types $(4)$, or $(3,2)$ or $(2,2,2)$ in a fiber of $\phi$. Another interesting possibility is a \emph{ramified node} (\autoref{fig:ramified_node}). It is a combination of multiple ramification and the development of a singularity. This is a node on $\orb C$, one of whose branches is simply ramified over $\orb P$, formally expressed by 
    \[ k\f{t} \to k\f{t,x}/x(x^2-t).\]
    
  \item \mbox{$\epsilon \leq 1/4$}

    In this case, $\orb C$ can have non-Gorenstein singularities. Indeed, the spatial triple point (formally the union of the coordinate axes in $\A^3$) is a branched cover of a line with branch divisor of multiplicity four. Since multiplicity four is allowed in the branch divisor for $\epsilon \leq 1/4$, the cover $\orb C \to \orb P$ can have formal local picture of a spatial triple point:
    \[ k\f{t} \to k\f{t,x,y}/(xy,y(x-t),x(y-t)).\]
  \end{compactenum}
\end{example}

In the case of admissible covers ($\epsilon = 1$) and in the case of hyperelliptic curves ($d = 2$), the branch morphism is finite. This is no longer the case if $d \geq 3$ and $\epsilon$ is sufficiently small. In fact, as soon as $\epsilon \leq 1/6$, we have positive dimensional fibers, as illustrated in the next example.
\begin{example}[Non-finiteness of the branch morphism]
\label{ex:non_finiteness}
  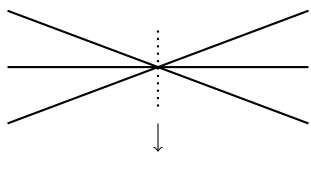
\begin{figure}[t]
    \centering
    \begin{tikzpicture}[thick, yscale=.75]
      \draw 
      (-2,-1) -- (2,1)
      (-2, 0) -- (2,0)
      (-2, 1) -- (2,-1);
      \draw[dotted] (0,-.7) -- (0,.7);
      \draw (-2,-2) -- (2,-2);
      \path[thin, ->] (0,-1) edge (0,-1.5);
    \end{tikzpicture}
    \caption{Planar triple points are allowed for $\epsilon \leq 1/6$}
    \label{fig:planar_triple_point}
  \end{figure}  

  For every $c \in k$, consider the the planar triple point expressed as a triple cover of a smooth curve (\autoref{fig:planar_triple_point}) by the formal description:
  \begin{equation}\label{eqn:planar_triple_points}
    k\f{t} \to k\f{t,x}/x(x-t)(x-ct).
  \end{equation}

  The discriminant is the ideal $\ideal{t^6}$. Although the rings $k\f{t,x}/x(x-t)(x-ct)$ are isomorphic for different choices of $c$, they are \emph{not} necessarily isomorphic as $k\f{t}$ algebras. Said differently, although the singularities $\spec k\f{t,x}/x(x-t)(x-ct)$ are isomorphic abstractly, they are \emph{not} necessarily isomorphic as triple covers of $\spec k\f{t}$. One way to see this is the following. Consider the tangent space to $\spec k\f{t,x}/x(x-t)(x-ct)$ at $(0,0)$. In this two dimensional vector space, there are four distinguished one dimensional subspaces: the three tangent spaces of the branches and the kernel of the projection to the tangent space of $\spec k\f{t}$. The moduli of the configuration of these four subspaces depends on $c$. Up to a finite ambiguity, different choices of $c$ give non-isomorphic triple covers.

  For $d \geq 3$, $\epsilon \leq 1/6$ and $h$, $b$ large enough to allow $\epsilon \cdot b + 2h-2 \geq 0$, the formal descriptions in \eqref{eqn:planar_triple_points} are realizable in covers of a fixed genus $h$ curve with a fixed branch divisor. We thus get infinitely many points in a fiber of $\br \from \o{\orb H}_{g/h}^d(\epsilon) \to \o{\orb M}_{h;b}(\epsilon)$.
\end{example}

In the case of admissible covers ($\epsilon = 1$), the small Hurwitz space ${\orb H}_{g/h}^d$ is dense in $\o{\orb H}_{g/h}^d(\epsilon)$. By \autoref{thm:smoothness_for_d23}, this remains the case for arbitrary $\epsilon$ if $d \leq 3$. However, this is not true in general, as illustrated by the following example.
\begin{example}[Extraneous components in $\o{\orb H}_{g;h}^d(\epsilon)$]\label{ex:extra_components}
  For a sufficiently large $d$ and a sufficiently small $\epsilon$, we exhibit a point in $\o{\orb H}_{g/h}^d(\epsilon)$ that is not in the closure of $\orb H^d_{g/h}$. For simplicity, take $h = 0$; the phenomenon is local, so the case of $h = 0$ can be used to construct examples for any $h$.
  
  Let $C$ be a reduced, connected curve that is not a flat limit of smooth curves (see the article by \citet{mumford75:_pathol_iv} for the existence of such curves). For sufficiently large $d$, we have a finite map $\phi \from C \to \P^1$ of degree $d$. Let $\epsilon$ be so small that $\epsilon \cdot \mult_p(\br\phi) \leq 1$ for all $p \in \P^1$. Then $(\P^1; \phi \from C \to \P^1)$ is a point in $\o{\orb H}_{g}^d(\epsilon)$ which, by construction, is not in the closure of $\orb H_g^d$. Thanks to \autoref{thm:smoothness_for_d23}, there are no extraneous components for $d = 2$ or $3$. By \autoref{thm:locally_smooth_implies_smooth}, unsmoothable singularities are the only reason for extraneous components. 
\end{example}

We end the section with a question prompted by \autoref{ex:extra_components}.
\begin{question}
  For which $d$, $g$ and $h$ is $\st H_{g/h}^d$ irreducible? More precisely, for which $d$, $g$, $h$ and $\epsilon$ is $\o{\orb H}_{g/h}^d(\epsilon)$ irreducible?
\end{question}


 \section{Moduli of $d$-gonal Singularities and Crimping}\label{sec:crimps}
 The goal of this section is to understand the fibers of $\br \from \st H^d \to \st M$. Consider a point $s \from \spec k \to \st M$. For  simplicity, assume that it corresponds to a smooth curve $P$ with a marked divisor $\Sigma$. The fiber of $\br$ over $s$ consists precisely of degree $d$ covers $\phi \from C \to P$ with $\br\phi = \Sigma$. Let $\tw{C} \to C$ be the normalization. Since $\tw{C}$ is smooth, the cover $\tw{C} \to P$ is determined by its restriction $\tw{C}|_{P\setminus\Sigma} \to P\setminus\Sigma$, which is \'etale. Since there are only finitely many \'etale covers of degree $d$ of a smooth curve, there are only finitely many possibilities for $\tw\phi \from \tw{C} \to P$. The fiber of $\br$ over $s$ thus decomposes into finitely many (open and closed) components corresponding to the choice of $\tw\phi \from \tw{C} \to P$. Within each component, $C \to P$ is obtained by \emph{crimping} a fixed $\tw{C} \to P$ over the points of $\Sigma$. The crimping can be described formally locally around the points of $\Sigma$ in $P$. In this way, the description of the fibers of $\br$ includes the discrete global data of the normalization and the continuous local data of the crimping. 

Moduli of singular curves and the phenomenon of crimping have been studied extensively by \citet{van_der_wyck:2010}. Our study of crimping in the context of finite covers, however, is much more elementary.

\subsection{The space of crimps of a finite cover}\label{sec:finite_cover_crimps}
Let $\orb Y$ be a reduced, purely one dimensional Deligne--Mumford stack over $k$ and $\Sigma \subset \orb Y$ a Cartier divisor. Let $\tw\phi \from \tw{\orb X} \to \orb Y$ a finite cover of degree $d$, \'etale over $\orb Y \setminus \Sigma$. In all the cases we consider, $\orb Y$ is either a (pointed) orbinodal curve or the spectrum of a DVR. Define the functor $\Crimp_{\tw\phi,\Sigma} \from \cat{Schemes}_k \to \cat{Sets}$ of \emph{crimps of $\tw\phi$ over $\Sigma$} by
\[ \Crimp_{\tw\phi,\Sigma}(T) = \{(\tw{\orb X} \times T \to \orb X \stackrel{\phi}{\to} \orb Y \times T)\}/\text{Isomorphism},\]
where $\phi \from \orb X \to \orb Y \times T$ is a finite cover of degree $d$ with $\br(\phi) = \Sigma \times T$. Two such crimps $\tw{\orb X} \times T \to \orb X_i \to \orb Y \times T$, for $i=1,2$, are isomorphic if there is an isomorphism $\orb X_1 \to \orb X_2$ that commutes with the relevant maps
\[
\begin{tikzpicture}
  \matrix (m) [matrix of math nodes, row sep=1em, column sep=2em]{
    \tw{\orb X} \times T & \orb X_1 & \orb Y \times T \\
    \tw{\orb X} \times T & \orb X_2 & \orb Y \times T \\
  };
  \path [map, ->]
  \foreach \i in {1,2}{
    (m-\i-1) edge (m-\i-2)
    (m-\i-2) edge (m-\i-3)
  }
  (m-1-2) edge (m-2-2);
  
  \draw [equal]
  (m-1-1) -- (m-2-1)
  (m-1-3) -- (m-2-3);
\end{tikzpicture}.
\]
We sometimes write $\Crimp(\tw\phi, \Sigma)$ instead of $\Crimp_{\tw\phi, \Sigma}$ for better readability.

If $\orb Z \to \orb Y$ is a morphism such that $\Sigma_{\orb Z} \subset \orb Z$ is also a divisor, then we have a natural transformation
\[ \Crimp(\tw\phi, \Sigma) \to \Crimp(\tw\phi_{\orb Z}, \Sigma_{\orb Z})\]
defined by
\[ (\tw{\orb X} \times T \to \orb X \stackrel{\phi}{\to} \orb Y \times T) \mapsto (\tw{\orb X}_{\orb Z} \times T \to \orb X_{\orb Z} \stackrel{\phi_{\orb Z}}{\to} \orb Y_{\orb Z} \times T).\]

Let $G = \Aut(\tw\phi)$ be the group of automorphisms of $\tw{\orb X}$ over the identity of $\orb Y$. This is a finite group, which acts on $\Crimp(\tw\phi, \Sigma)$ as follows:
\[ G \ni \alpha \from (\tw{\orb X} \times T \stackrel{\nu}\longrightarrow \orb X \stackrel{\phi}{\to} \orb Y \times T) \mapsto (\tw{\orb X} \times T \stackrel{\nu\circ\alpha^{-1}}\longrightarrow \orb X \stackrel{\phi}{\to} \orb Y \times T).\]

\begin{remark}\label{rem:crimp_subalgebra}
A crimp may be equivalently thought of as a suitable subalgebra $\phi_*O_{\orb X}$ of the algebra $\tw\phi_*O_{\tw{\orb X} \times T}$ on $\orb Y \times T$. Then isomorphism of crimps simply becomes equality of subalgebras. The action of $G$ is induced by the action of $G$ on $\tw\phi_* O_{\tw{\orb X}}$.
\end{remark}
Throughout, we view $O_{\tw{\orb X} \times T}$ and $O_{\orb X}$ as sheaves of algebras on $\orb Y \times T$, omitting $\tw\phi_*$ and $\phi_*$ to lighten notation. Observe that the quotient $O_{\tw{\orb X} \times T}/O_{\orb X}$ is an $O_{\orb Y\times T}$ module supported entirely on $\Sigma \times T$. In other words, $\tw {\orb X} \times T \to \orb X$ is an isomorphism away from $\Sigma \times T$.

Having defined $\Crimp(\tw\phi, \Sigma)$ in wide generality, we turn to the case of interest. Let $(\orb P \to P; \sigma_1, \dots, \sigma_n)$ be a pointed orbinodal curve and $\Sigma \subset \orb P$ a divisor supported in the general locus $P^\gen = \orb P^\sm \setminus {\sigma_1, \dots, \sigma_n}$. Let $\tw\phi \from \tw{\orb C} \to \orb P$ be a finite cover, \'etale over $\orb P \setminus \Sigma$. We begin by making precise our remark that crimps can be described formally locally around the points of $\Sigma$. 
\begin{proposition}\label{thm:crimp_local}
  Let $\tw \phi \from \tw{\orb C} \to \orb P$ and $\Sigma$ be as above.
  \begin{enumerate}
  \item Let $U \subset \orb P$ be an open set containing $\Sigma$. Then the transformation
    \[ \Crimp(\tw\phi, \Sigma) \to \Crimp(\tw\phi_U, \Sigma)\]
    is an isomorphism.
  \item The transformation
    \[ \Crimp(\tw\phi, \Sigma) \to \prod_{s \in \supp(\Sigma)} \Crimp(\tw \phi \times_{\orb P} \spec O_{P,s}, \Sigma \times_{\orb P}\spec O_{P,s})\]
    is an isomorphism.
  \item\label{eqn:crimp_formal} The transformation
    \[ \Crimp(\tw\phi, \Sigma) \to \prod_{s \in \supp(\Sigma)} \Crimp(\tw \phi \times_{\orb P} \spec \compl{O_{P,s}}, \Sigma \times_{\orb P}\spec \compl{O_{P,s}})\]
    is an isomorphism.
  \end{enumerate}
\end{proposition}
\begin{proof}
  The last assertion is the strongest, so we prove that. Following \autoref{rem:crimp_subalgebra}, we treat crimps as subalgebras. For brevity, we set
    \[
      \compl{P}_s = \spec \compl{O_{P,s}}, \quad
      \Sigma_s = \Sigma \times_{P}\compl{P}_s, \text{ and }
      \tw{C}_s = \tw{\orb C} \times_{\orb P} \compl{P}_s,
      \]
      Given crimps $\tw{C}_s \times T \to C_s \to \compl{P}_s\times T$ for $s \in \supp(\Sigma)$, construct a subalgebra $O_{\orb C}$ of $O_{\tw{\orb C} \times T}$ as the fiber product of algebras
      \[
      \begin{tikzpicture}
        \matrix (m) [matrix of math nodes, row sep=1em, column sep=4em]{
          O_{\orb C} & O_{\tw{\orb C} \times T}\\
          \prod_s O_{C_s} & \prod_s O_{\tw C_s \times T}\\
        };
        \path [->]
        (m-1-1) edge (m-1-2)
        (m-1-1) edge (m-2-1)
        (m-1-2) edge (m-2-2)
        (m-2-1) edge (m-2-2);
      \end{tikzpicture}.
    \]
    We thus get a natural transformation 
    \[ \prod_{s \in \supp(\Sigma)} \Crimp(\tw \phi \times_{\orb P} \spec \compl{O_{P,s}}, \Sigma \times_{\orb P}\spec \compl{O_{P,s}}) \to \Crimp(\tw\phi, \Sigma).\]
    It is easy to check that it is inverse to the transformation in \eqref{eqn:crimp_formal}.
\end{proof}

\subsection{Crimps over a disk}\label{sec:crimps_disk}
Thanks to \autoref{thm:crimp_local}, we now focus on the crimps of covers of the formal disk. Set $R = k\f{t}$ and $\Delta = \spec R$. Let $\Delta^\circ$ be the punctured disk $\Delta\setminus\{0\}$. Fix a finite cover $\tw\phi \from \tw C \to \Delta$ of degree $d$, \'etale over $\Delta^\circ$, with $\br(\tw\phi)$ given by $\ideal{t^a}$. Fix a divisor $\Sigma \subset \Delta$ given by $\ideal{t^b}$ and set $\delta = (b-a)/2$.

\begin{proposition}\label{thm:delta_inv}
  Let $\tw C \times T \to C \stackrel\phi\to \Delta \times T$ be a crimp with $\br(\phi) = \Sigma \times T$. Set $Q = O_{\tw C \times T}/O_C$. Then $Q$ is a $T$-flat sheaf on $\Delta \times T$ annihilated by $t^b$. The restriction of $Q$ to the fibers of $\Delta \times T \to T$ has length $\delta$.
\end{proposition}
\begin{proof}
  In the proof, all the linear-algebraic operations are over $O_{\Delta \times T}$.  First, $Q$ is $T$-flat simply because the inclusion $i \from O_C \into O_{\tw C \times T}$ remains an inclusion when restricted to the fibers of $\Delta \times T \to T$. For the rest, consider the diagram
  \[
  \begin{tikzpicture}
    \matrix (m) [matrix of math nodes, row sep=1.5em, column sep=1em]{
      0 & O_{\Delta \times T} & (\det \dual{O_{\tw C \times T}})^{\otimes 2} & \tw B & 0\\
      0 & O_{\Delta \times T} & (\det \dual{O_C})^{\otimes 2}& B & 0\\
    };
    \path[->, math, map] 
    (m-1-2) edge node[auto]{\tw\delta} (m-1-3)
    (m-2-2) edge node[auto]{\delta} (m-2-3)
    \foreach \i in {1,2}{
      (m-\i-1) edge (m-\i-2)
      (m-\i-3) edge (m-\i-4) 
      (m-\i-4) edge (m-\i-5)
    }
      (m-1-3) edge node[auto] {\det(\dual{i})^2} (m-2-3)
      (m-1-4) edge (m-2-4);
    \draw [equal] (m-1-2) -- (m-2-2);
  \end{tikzpicture}.
  \]
  The horizontal maps $\tw \delta$ and $\delta$ define the respective branch divisors as in \autoref{sec:Ad}. In particular, $B$ is annihilated by $\ideal{t^b}$. The snake lemma yields the sequence
  \begin{equation}\label{eqn:bound_delta}
    0 \to \tw B \to B \to \cok(\det(\dual{i})^2) \to 0.
  \end{equation}
Since $t^b$ annihilates $B$, it annihilates $\cok(\det(\dual{i})^2)$, hence $\cok(\det(\dual{i}))$, hence $\cok(\dual{i})$ and hence $\cok i = Q$.

To compute the length of $Q$ on the fibers, replace $T$ by a field. By \eqref{eqn:bound_delta}, we get
  \begin{align*}
    2\length Q &= \length (\cok (\det(\dual{i})^2))\\
    &= \length B- \length \tw B \\
    &= b-a = 2\delta.
  \end{align*}
\end{proof}
\begin{remark}
  By \autoref{thm:delta_inv}, if $\tw C$ is smooth, then $\delta$ is indeed the $\delta$ invariant of $C$.
\end{remark}

 We now exhibit the space of crimps over a disk explicitly as a projective variety. Set $F = O_{\tw C}/t^b O_{\tw C}$ and denote by $\Quot = \Quot(F, \delta)$ the Quot scheme of length $\delta$ quotients of the $\Delta$ module $F$. Since $\supp F$ is projective (it is finite!), $\Quot$ is a projective scheme. The idea is to identify quotients which arise as $O_{\tw C}/O_C$. For this to be true, the quotient must satisfy the following two properties:
 \begin{compactenum}
 \item\label{eqn:subalg} The kernel must be closed under multiplication, to get a subalgebra $O_C$ of $O_{\tw C}$;
 \item\label{eqn:branch} The resulting $C \to \Delta$ must have the right branch divisor.
 \end{compactenum}
 We now formalize both conditions. Let $\pi \from \Delta \times \Quot \to \Delta$ be the projection. On $\Delta \times \Quot$ we have the universal sequence
 \[ 0 \to S \to F \otimes_k O_{\Quot} \to Q \to 0.\]
 The multiplication $F \otimes_\Delta F \to F$ induces maps
 \[ S \otimes_{\Delta \times \Quot} S \to (F \otimes_{\Delta} F) \otimes_k O_{\Quot} \to F \otimes_k O_{\Quot} \to Q.\]
 Define the closed subscheme $X \subset \Quot$ as the annihilator of the composite map $\pi_*(S\otimes_{\Delta \times \Quot} S) \to \pi_*Q$ on $\Quot$. This takes care of \eqref{eqn:subalg}.

 On $\Delta \times X$, the sheaf $S$ inherits the structure of an $O_{\Delta \times X}$ algebra. Form the subalgebra $O_C$ of $O_{\tw C \times X}$ as the fiber product 
 \[
 \begin{tikzpicture}
   \matrix (m) [matrix of math nodes, row sep=1em, column sep=4em]{
     O_C & O_{\tw C \times X}\\
     S & F \otimes_k O_{X}\\
   };
   \path [->]
   (m-1-1) edge (m-1-2)
   (m-1-1) edge (m-2-1)
   (m-1-2) edge (m-2-2)
   (m-2-1) edge (m-2-2);
 \end{tikzpicture},
 \]
 and set $C = \spec O_C$. 
\begin{claim}
  In the above setup, $C \to \Delta \times X$ is flat.
\end{claim}
\begin{proof}
  By the definition of $O_C$, we have the sequence
 \[ 0 \to O_C \to O_{\tw C \times X} \to Q \to 0.\]
 Since $Q$ is $X$-flat, we conclude that $O_C$ is $X$-flat and $O_C \to O_{\tw C \times X}$ remains an inclusion when restricted to the fibers of $\Delta \times X \to X$. For every point $x \in X$, the sheaf $O_{C_x}$ is a subsheaf of the free sheaf $O_{\tw C}$ and hence is free. It follows that $O_C$ is a locally free $\Delta \times X$ module.
\end{proof}

We currently have $\tw C \times X \to C \stackrel\phi\to \Delta \times X$, where $\tw C \to C$ is an isomorphism over $\Delta^\circ \times X$ and $C \to \Delta \times X$ is finite and flat. We now enforce \eqref{eqn:branch}. Define $B$ by
 \[ 0 \to O_{\Delta \times X} \stackrel\delta\to (\det \dual{O_C})^{\otimes 2} \to B \to 0,\]
 where the linear algebraic operations are over $O_{\Delta \times X}$, and $\delta$ is the usual discriminant as in \autoref{sec:Ad}. Observe that $\delta$ remains an injection when restricted to the fibers of $\pi \from \Delta \times X \to X$, and hence $B$ is $X$-flat. See that $B$ has fiberwise length $b$. Define the closed subscheme $Y \subset X$ as the annihilator of 
 \[ \pi_* B \stackrel{t^b}\longrightarrow \pi_* B.\]
This condition would be superfluous if $X$ were reduced. However, it appropriately restricts the non-reduced structure on $X$, taking care of \eqref{eqn:branch}. 

By construction, we have a crimp $\tw C \times Y \to C \stackrel\phi\to \Delta \times Y$ with $\br\phi = \Sigma \times Y$. We thus get a morphism
\begin{equation}\label{eqn:crimp_projective}
 Y \to \Crimp(\tw\phi, \Sigma).
\end{equation}

\begin{proposition}\label{thm:crimps_over_disk}
  The morphism $ Y \to \Crimp(\tw\phi, \Sigma)$ in \eqref{eqn:crimp_projective} is an isomorphism. In particular, $\Crimp(\tw \phi, \Sigma)$ is a projective scheme.
\end{proposition}
\begin{proof}
  We construct the inverse $\Crimp(\tw \phi, \Sigma) \to Y$ to \eqref{eqn:crimp_projective}. Let $T$ be a scheme and $\tw C \times T \to C \stackrel\phi\to \Delta \times T$ a crimp with branch divisor $\Sigma \times T$. Define the quotient $Q = O_{\tw C \times T}/O_C$. By \autoref{thm:delta_inv}, $Q$ is a $T$-flat quotient of $O_{\tw C \times T}/t^b O_{\tw C \times T} = F \otimes_k O_T$, fiberwise of length $\delta$. This gives a map $T \to \Quot(F, \delta)$. Since the kernel of $F \otimes_k O_T \to Q$ is the image of $O_C$, it is closed under multiplication. Hence $T \to \Quot$ factors through $T \to X$. Since $\br(\phi) = \Sigma \times T$, the cokernel of 
  \[O_{\Delta \times T} \stackrel\delta\to (\det \dual{O_{C}})^{\otimes 2}\]
  is annihilated by $t^b$. Therefore, the map $T \to X$ factors through $T \to Y$. In this way, we get a morphism $\Crimp(\tw\phi, \Sigma) \to Y$, which is clearly inverse to \eqref{eqn:crimp_projective}.
\end{proof}
\begin{corollary}\label{thm:crimps_projective}
  Let $\tw{\orb C} \to \orb P$ be a finite cover of an orbinodal curve and $\Sigma \subset P^\gen$ a divisor. Then the functor $\Crimp(\tw{\orb C} \to \orb P, \Sigma)$ is representable by a projective scheme.
\end{corollary}
\begin{proof}
  Follows immediately from \autoref{thm:crimp_local} and \autoref{thm:crimps_over_disk}.
\end{proof}

Finally, we relate the spaces of crimps with the fibers of $\br \from \st H^d \to \st M$. Let $p \from \spec k \to \st M$ be a point corresponding to a divisorially marked, pointed curve $(P; \Sigma; \sigma_1, \dots, \sigma_n)$. As usual, we abbreviate $\sigma_1, \dots, \sigma_n$ by $\sigma$. Let $\Gamma$ be the set of $(\orb P \to P; \sigma; \tw\phi \from \tw{\orb C} \to \orb P)$, where $(\orb P \to P; \sigma)$ is a pointed orbinodal curve and $\tw\phi$ a finite cover of degree $d$ such that
\begin{compactenum}
\item $\tw{\orb C} \times_{\orb P}{\orb P}^\sm$ is smooth;
\item $\tw\phi$ is \'etale over $\orb P \setminus \Sigma$; and
\item $\tw\phi$ corresponds to a representable classifying map $\orb P \to \st A_d$.
\end{compactenum}
Assume that no two elements of $\Gamma$ are isomorphic over the identity of $P$. Then $\Gamma$ is a finite set. We have a morphism
\begin{equation}\label{eqn:crimp_to_br_fiber}
  \bigsqcup_\Gamma \Crimp(\tw\phi, \Sigma) \to  p \times_{\st M} \st H^d.
\end{equation}
given by
\[ (\tw {\orb C} \times T \to \orb C \stackrel\phi\to \orb P \times T) \mapsto (\orb P \times T \to P \times T; \sigma \times T; \orb C \stackrel\phi\to \orb P \times T).\]
Recall that we have an action of $\Aut(\tw\phi)$ on $\Crimp(\tw\phi, \Sigma)$. The morphism above clearly descends to a morphism
\begin{equation}\label{eqn:crimp_to_br_fiber_mod_automorphisms}
  \bigsqcup_\Gamma [\Crimp(\tw\phi, \Sigma)/\Aut(\tw\phi)] \to  p \times_{\st M} \st H^d.
\end{equation}

\begin{proposition}\label{thm:crimp_fiber}
  The morphism in \eqref{eqn:crimp_to_br_fiber} is finite and surjective. The morphism in \eqref{eqn:crimp_to_br_fiber_mod_automorphisms} is representable and a bijection on $k$-points.
\end{proposition}
\begin{proof}
  The statement is true almost by design. Nevertheless, here are the details. Let $s \from \spec k \to p \times_{\st M}\st H^d$ be a point given by $(\orb P' \to P; \sigma; \orb C \to \orb P')$. We first check that the fiber of \eqref{eqn:crimp_to_br_fiber} over $s$ is nonempty and forms one orbit under the group action. Let $\orb C' \to \orb C$ be the partial normalization obtained by normalizing $\orb C$ away from its nodes over the nodes of $\orb P'$. Then $(\orb P' \to P; \sigma; \orb C' \to \orb P')$ is isomorphic to some
    $(\orb P \to P; \sigma; \tw{\orb C} \to \orb P)$ in $\Gamma$. Identify $\orb P'$ and $\orb P$ via an isomorphism $\orb P \isom \orb P'$ over the identity of $P$.
    
    For every choice of isomorphism $\tw{\orb C} \to \orb C'$ over $\orb P$, we have a point $\tw{\orb C} \to \orb C \to \orb P$ of $\Crimp(\tw{\orb C} \to \orb P, \Sigma)$ lying over $s$. Conversely, it is clear these are exactly the points in the fiber of \eqref{eqn:crimp_to_br_fiber} over $s$. We conclude that \eqref{eqn:crimp_to_br_fiber} is finite, surjective  and \eqref{eqn:crimp_to_br_fiber_mod_automorphisms} is a bijection on $k$ points.

Finally, suppose we have a non-trivial automorphism $\tw \phi \from \tw{\orb C}\to \tw {\orb C}$ over the identity of $\orb P$ that induces an automorphism $\phi \from \orb C \to \orb C$. Then, clearly, $\phi$ is non-trivial. Hence \eqref{eqn:crimp_to_br_fiber_mod_automorphisms} is representable.
\end{proof}

\autoref{thm:crimp_fiber} is as close as we can come to explicitly identifying the fibers of $\br \from \st H^d \to \st M$. This is good enough for determining many crude properties like dimension.


\section*{Acknowledgements}
  This work is a part of my PhD thesis. I am deeply grateful to my adviser Joe Harris for his invaluable insight and immense generosity. My heartfelt thanks to Maksym Fedorchuk for inspiring this project and providing guidance at all stages. I thank Dan Abramovich, Brendan Hassett, Anand Patel, David Smyth and Ravi Vakil for valuable suggestions and conversations.

\bibliographystyle{abbrvnat}
\bibliography{CommonMath}
\end{document}